\RequirePackage{fix-cm}
\documentclass[smallextended]{svjour3}       % onecolumn (second format)
\smartqed  % flush right qed marks, e.g. at end of proof
\usepackage{graphicx}
\usepackage{enumerate}
\usepackage{color}
\usepackage{amsfonts,amssymb} %%%add
\usepackage{amsmath} %%%add
\begin{document}

\title{Sharp error estimates for spatial-temporal finite difference approximations to fractional sub-diffusion equation without regularity assumption on the exact solution
	\thanks{This work was supported by the National Natural Science Foundation of China under Grant Nos. 12071195 and 12201270, the Innovative Groups of Basic Research in Gansu Province under Grant No.22JR5RA391, the science and technology plan of Gansu Province under Grant No. 22JR5RA535, the Fundamental Research Funds for the Central Universities under Grant No. lzujbky-2022-pd04, and China Postdoctoral
		Science Foundation: 2022M721439.
	}%\thanks{Grants or other notes
}

\titlerunning{Error Estimates of Finite Difference  Approximations}        % if too long for running head

\author{Daxin Nie         \and
		Jing Sun\and
        Weihua Deng %etc.
}


\institute{Daxin Nie \at
              School of Mathematics and Statistics, Gansu Key Laboratory of Applied Mathematics and Complex Systems, Lanzhou University, Lanzhou 730000, P.R. China \\
              \email{ndx1993@163.com}           %  \\
           \and
           Jing Sun \at
           School of Mathematics and Statistics, Gansu Key Laboratory of Applied Mathematics and Complex Systems, Lanzhou University, Lanzhou 730000, P.R. China \\
           \email{js@lzu.edu.cn}           %  \\
           \and
           Weihua Deng \at
              School of Mathematics and Statistics, Gansu Key Laboratory of Applied Mathematics and Complex Systems, Lanzhou University, Lanzhou 730000, P.R. China\\
              \email{dengwh@lzu.edu.cn}
}

\date{Received: date / Accepted: date}

\maketitle

\begin{abstract}
Finite difference method as a popular numerical method has been widely used to solve fractional diffusion equations. In the general spatial error analyses, an assumption $u\in C^{4}(\bar{\Omega})$ is needed to preserve $\mathcal{O}(h^{2})$ convergence when using  central finite difference scheme to solve fractional sub-diffusion equation with Laplace operator, but this assumption is somewhat strong, where $u$ is the exact solution and $h$ is the mesh size. In this paper, a novel analysis technique is proposed to show that the spatial convergence rate can reach $\mathcal{O}(h^{\min(\sigma+\frac{1}{2}-\epsilon,2)})$ in both $l^{2}$-norm and $l^{\infty}$-norm in one-dimensional domain when the initial value and source term are both in $\hat{H}^{\sigma}(\Omega)$ but without any regularity assumption on the exact solution, where  $\sigma\geq 0$ and $\epsilon>0$ being arbitrarily small. After making slight modifications on the scheme, acting on the initial value and source term, the spatial convergence rate can be improved to $\mathcal{O}(h^{2})$ in $l^{2}$-norm and  $\mathcal{O}(h^{\min(\sigma+\frac{3}{2}-\epsilon,2)})$ in $l^{\infty}$-norm. It's worth mentioning that our spatial error analysis is applicable to high dimensional cube domain by using the properties of tensor product.  Moreover, two kinds of averaged schemes are provided to approximate the Riemann--Liouville fractional derivative, and $\mathcal{O}(\tau^{2})$ convergence is obtained for all $\alpha\in(0,1)$. Finally, some numerical experiments verify the effectiveness of the built theory.
\keywords{Fractional sub-diffusion equation\and Laplace operator\and central finite difference method\and modified scheme\and  averaged $L1$ scheme\and averaged second order backward difference scheme\and error analysis}
\end{abstract}

\section{Introduction}

In this paper, we mainly propose sharp error estimates for the finite difference scheme of the fractional sub-diffusion equation without regularity assumption on the exact solution, i.e.,
\begin{equation}\label{eqretosol}
	\left\{
	\begin{aligned}
		&{}_{0}\partial^{\alpha}_{t}(u-u_{0})+Au=f\qquad (x,t)\in \Omega\times(0,T],\\
		&u=0\qquad (x,t)\in \partial\Omega\times(0,T],\\
		&u(0)=u_{0}\qquad x\in \Omega,
	\end{aligned}
	\right.
\end{equation}
where $\Omega=(0,l)^{d}\in \mathbb{R}^{d}$, and without loss of generality, we take $d=l=1$; $A=(-\Delta)$ with homogeneous Dirichlet boundary condition, whose eigenvalues and eigenfunctions are $\{\lambda_{k},\phi_{k}\}_{k=1}^{\infty}=\{k^{2}\pi^{2},\sqrt{2}\sin(k\pi x)\}_{k=1}^{\infty}$, and the eigenfunctions are orthogonal in $L^{2}(\Omega)$; $f$ is the source term; $u_{0}$ is the initial value; ${}_{0}\partial^{\alpha}_{t}$ is the Riemann--Liouville fractional derivative \cite{Podlubny.1999FDE}, the definition of which is
\begin{equation*}
	{}_{0}\partial^{\alpha}_{t}u=\frac{1}{\Gamma(1-\alpha)}\frac{\partial}{\partial t}\int_{0}^{t}(t-\xi)^{-\alpha}u(\xi)d\xi.
\end{equation*}

Fractional sub-diffusion equation \eqref{eqretosol} governs the evolution of the probability density function of the sub-diffusion \cite{Barkai.2001FFesaa,Metzler.1999AdarctteAfFea} and there are many applications in physics, biology, and so on  \cite{Glockle.1995Afcatspd,Zaslavsky.1997SrapsnoHcd}. So far, there have been many works on numerically solving fractional sub-diffusion equations \cite{Fu.2021EssfdAmfmvtde,Gao.2015Sacofdsfacotsebocs,Jin.2015AaotLsftsewnd,Jin.2017CohBcqffee,Lin.2007Fdafttde,Mustapha.2020Aafafreaseaotm,Stynes.2017Eaoafdmogmfatde,Wang.2020THTDSfSPwND,Yan.2018AaotmLsftpdewnd,Zeng.2015SSFDSftTDE,Zheng.2022AaLcdmftmdewwss}. Among them, the central finite difference scheme is a popular approximation to Laplace operator in  Eq. \eqref{eqretosol} and it can theoretically achieve $\mathcal{O}(h^{2})$ under the regularity assumption of the exact solution $u\in C^{4}(\bar{\Omega})$, which implies the initial value and source term must be smooth enough, such as, belong to $\hat{H}^{5}(\Omega)$ (the corresponding regularity analyses can refer to \cite{Sakamoto.2011Ivvpffdeaatsip,Stynes.2017Eaoafdmogmfatde}). However, are these assumptions really necessary?

In recent years,  a variety of robust numerical schemes for solving the fractional partial differential equations with non-smooth data and the corresponding error analyses  have been proposed \cite{Acosta.2019Feaffep,Deng.2018TDoaTFFEwMD,Jin.2016Tfdsffdadewnd,Jin.2019SwatcAans,Li.2022Ecqfnsewnid,Luo.2019CAoaPMfFWPwND}. To our best knowledge, it seems that the relative discussions on central finite difference scheme are rare. The main reason is that the existing error estimate for finite difference scheme is constructed by the truncation errors, which can't be obtained without the regularity assumption of the exact solution.

To fill the gap, a new error analysis without the regularity assumption on exact solution is proposed in this paper. To be specific, motivated by the idea in \cite{Gyongy.1998LafsqppdedbswnI,Gyongy.1999LafsqppdedbswnI}, we first give a representation of the solution of central finite difference scheme to Eq. \eqref{eqretosol}. Then by using approximation theory and operator theory, an $\mathcal{O}(h^{\min(\sigma+\frac{1}{2}-\epsilon,2)})$ spatial convergence rate in both $l^{2}$-norm and $l^{\infty}$-norm can be obtained if $u_{0},f(0)\in \hat{H}^{\sigma}(\Omega)$ and $\int_{0}^{t}\|f_{t}(s)\|_{\hat{H}^{\sigma}(\Omega)}ds<\infty$ (for the details, see Theorems \ref{thmerrorhom} and \ref{thmerrorimhom}). Moreover, modifying the scheme, i.e., slightly changing the initial value and source term makes spatial convergence rates  be improved to $\mathcal{O}(h^{2})$ in $l^{2}$-norm and $\mathcal{O}(h^{\min(\sigma+\frac{3}{2}-\epsilon,2)})$ in  $l^{\infty}$-norm. And the spatial error analyses for high dimensional cube domain can be similarly obtained by using the properties of tensor product, which can refer to Remark \ref{Rehighdimen}. On the other hand,  two kinds of averaged schemes, i.e., averaged $L1$ and averaged second order backward difference ($SBD$) schemes, are used to discretize the Riemann-Liouville fractional derivative, and an $\mathcal{O}(\tau^{2})$ convergence is obtained. Different from the convergence analysis of the previous averaged $L1$ scheme provided in \cite{Zheng.2022AaLcdmftmdewwss,Zhou.2022OCRiTDT$L$LaAS}, our temporal error estimate holds for all $\alpha\in(0,1)$.

The paper is organized as follows. We propose central finite difference scheme and the modified central difference scheme to approximate the Laplace operator, and the corresponding error analyses without regularity assumption  of the exact solution are provided in Section \ref{sec2}. In Section \ref{sec3}, two averaged schemes are constructed in time and we show that their convergence rates are both $\mathcal{O}(\tau^{2})$ for all $\alpha\in(0,1)$. Various numerical experiments are presented in Section \ref{sec4} to validate our theory. At last, we conclude the paper with some discussions.
In the following, $C$ is a positive constant, whose value may differ at each occurrence, $\epsilon>0$ is arbitrarily small and `$\tilde{~}$' stands for taking Laplace transform.
\section{Space semi-discrete scheme and error analysis}\label{sec2}
In this section, we approximate the Laplace operator in Eq. \eqref{eqretosol} by central finite difference method and provide the corresponding error analysis. Compared with the traditional error analysis, our built spatial error analysis only depends on  the regularity of the initial value and the source term, instead of the regularity of exact solution. Afterwards, we do some modifications on the scheme, which makes the spatial convergence rate be improved to $\mathcal{O}(h^{2})$ in $l^{2}$-norm and  $\mathcal{O}(h^{\min(\sigma+\frac{3}{2}-\epsilon,2)})$ in $l^{\infty}$-norm.

\subsection{Central finite difference scheme and error analysis}
%Let's begin with introducing some notations.
Define the fractional Sobolev space $\hat{H}^{s}(\Omega)$ with $s\geq 0$ by
\begin{equation*}
	\hat{H}^{s}(\Omega)=\{v\in L^{2}(\Omega), \|v\|_{\hat{H}^{s}(\Omega)}<\infty\},
\end{equation*}
whose norm can be defined by
\begin{equation*}
	\|v\|_{\hat{H}^{s}(\Omega)}^{2}=\sum_{k=1}^{\infty}\lambda_{k}^{s}(v,\phi_{k})^{2}.
\end{equation*}
Here $(\cdot,\cdot)$ denotes the $L^{2}$ inner product, and $\phi_{k}$ is the eigenfunction of the Laplace operator with homogeneous boundary condition on $\Omega$. Introduce the discrete $l^{2}$-norm and $l^{\infty}$-norm as
\begin{equation*}
	\|\mathbf{v}\|_{l^{2}}^{2}=\frac{1}{M}\sum_{k=1}^{M}v_{k}^{2},\quad \|\mathbf{v}\|_{l^{\infty}}=\max_{1\leq k\leq M}|v_{k}|,\quad{\rm for}~\mathbf{v}=[v_{1},v_{2},\ldots,v_{M}]^{T}\in \mathbb{R}^{M}.
\end{equation*}
Let the mesh size $h=1/N$, $N\in\mathbb{N}^{*}$, and $x_{i}=i/N$, $i=0,\ldots,N$. Define the operator $\mathcal{I}_{N}$ as
\begin{equation}\label{eqdefI}
	\mathcal{I}_{N}v=[v(x_{1}),v(x_{2}),\ldots,v(x_{N-1})]^{T}
\end{equation}
for a function $v$. Denote
\begin{equation}\label{eqdefbu}
	\begin{aligned}
		\mathbf{u}=&[u_{1},u_{2},\ldots,u_{N-1}]^{T}=\mathcal{I}_{N}u,\\
		\mathbf{u}^{0}=&[u^{0}_{1},u^{0}_{2},\ldots,u^{0}_{N-1}]^{T}=\mathcal{I}_{N}u_{0},\\
		\mathbf{f}=&[f_{1},f_{2},\ldots,f_{N-1}]^{T}=\mathcal{I}_{N}f.
	\end{aligned}
\end{equation}

Then, using the central finite difference method to discretize minus Laplace operator $-\Delta$ leads to the spatial semi-discrete scheme of Eq.  \eqref{eqretosol} as follows:
\begin{equation}\label{eqsemischl1}
	\begin{aligned}
		&{}_{0}\partial^{\alpha}_{t}(u_{h,i}-u^{0}_{i})+N^{2}\sum_{j=1}^{N-1}w_{i,j}u_{h,j}=f_{i}
	\end{aligned}
\end{equation}
with
\begin{equation*}
	w_{i,j}=\left\{
	\begin{array}{cl}
		2 & i=j,\\
		-1 & |i-j|=1,\\
		0 & |i-j|>1,
	\end{array}
	\right.
\end{equation*}
for $i,j=1,2,\ldots, N-1$. Here $u_{h,i}$ is the numerical solution of $u$ at $x_{i}$.
Let $\mathbf{u}_{h}=[u_{h,1},u_{h,2},\ldots,u_{h,N-1}]^{T}$ and the matrix $\mathbf{A}=[N^{2}w_{i,j}]_{i,j=1}^{N-1}$. Then Eq. \eqref{eqsemischl1}  becomes
\begin{equation}\label{eqsemischl1M}
	\begin{aligned}
		&{}_{0}\partial^{\alpha}_{t}(\mathbf{u}_{h}-\mathbf{u}^{0})+\mathbf{A}\mathbf{u}_{h}=\mathbf{f}.
	\end{aligned}
\end{equation}
Taking the Laplace transform and inverse Laplace transform for Eq. \eqref{eqsemischl1M},  the solution can be written as
\begin{equation}\label{eqsemischl1Msol0}
	\begin{aligned}
		\mathbf{u}_{h}(t)=\frac{1}{2\pi\mathbf{i}}\int_{\Gamma_{\theta}}e^{zt}z^{\alpha-1}(z^{\alpha}+\mathbf{A})^{-1}\mathbf{u}^{0}dz+\frac{1}{2\pi\mathbf{i}}\int_{\Gamma_{\theta}}e^{zt}(z^{\alpha}+\mathbf{A})^{-1}\tilde{\mathbf{f}}dz,
	\end{aligned}
\end{equation}
where $\Gamma_{\theta}=\{z\in\mathbb{C},  |\arg(z)|=\theta,|z|>0,\,\theta\in (\pi/2,\pi)\}$, $\arg(z)$ means the argument of $z$, and $\mathbf{i}$ denotes  the imaginary unit.

On the other hand, following \cite{Gyongy.1998LafsqppdedbswnI,Gyongy.1999LafsqppdedbswnI}, the eigenvectors and eigenvalues of matrix $\mathbf{A}$ are
\begin{equation*}
	\{\boldsymbol{\varphi}_{h,j}\}_{j=1}^{N-1}=\left \{\left[\sqrt{\frac{2}{N}}\sin\left (j\frac{k}{N}\pi\right )\right]_{k=1}^{N-1}\right \}_{j=1}^{N-1}
\end{equation*}
and
\begin{equation*}
	\{\lambda_{h,j}\}_{j=1}^{N-1}=\left\{4N^{2}\sin^{2}\left (\frac{j}{2N}\pi\right )\right\}_{j=1}^{N-1}=\{j^{2}\pi^{2}c^{N}_{j}\}_{j=1}^{N-1}
\end{equation*}
with
\begin{equation*}
	c^{N}_{j}=\frac{\left(\sin\left(\frac{j\pi}{2N}\right)\right)^{2}}{\left(\frac{j\pi}{2N}\right)^{2}}.
\end{equation*}
Thus $\mathbf{u}_{h}(t)$ can be represented by the eigenvectors of matrix $\mathbf{A}$, i.e.,
\begin{equation}\label{eqsemischl1Msol1}
	\begin{aligned}
		\mathbf{u}_{h}(t)=&\frac{1}{2\pi\mathbf{i}}\int_{\Gamma_{\theta}}e^{zt}z^{\alpha-1}(z^{\alpha}+\mathbf{A})^{-1}\sum_{i=1}^{N-1}(\mathbf{u}^{0})^{T}\boldsymbol{\varphi}_{h,i}\boldsymbol{\varphi}_{h,i}dz\\
		&+\frac{1}{2\pi\mathbf{i}}\int_{\Gamma_{\theta}}e^{zt}(z^{\alpha}+\mathbf{A})^{-1}\sum_{i=1}^{N-1}(\tilde{\mathbf{f}})^{T}\boldsymbol{\varphi}_{h,i}\boldsymbol{\varphi}_{h,i}dz.
	\end{aligned}
\end{equation}
Further, define $u_{h,0}(t)=u_{h,N}(t)=0$, introduce $u_{h}(x,t)$ and $\phi_{h,j}(x)$ as
\begin{equation}\label{eqdefuh}
	\begin{aligned}
		&u_{h}(x,t)=u_{h,i}+(Nx-i)(u_{h,i+1}-u_{h,i})\qquad {\rm} x\in[x_{i},x_{i+1}),~i=0,1,\ldots,N-1,\\
		&\phi_{h,j}(x)=\phi_{j}(x_{i})+(Nx-i)(\phi_{j}(x_{i+1})-\phi_{j}(x_{i}))\qquad {\rm} x\in[x_{i},x_{i+1}),~i=0,1,\ldots,N-1,
	\end{aligned}
\end{equation}
and denote $\tilde{E}_{h}(z,x,y)$ as
\begin{equation}\label{eqdefEh}
	\begin{aligned}
		&\tilde{E}_{h}(z,x,y)=\sum_{i=1}^{N-1}(z^{\alpha}+\lambda_{h,i})^{-1}\phi_{h,i}(x)\phi_{i}(\eta_{h}(y)),
	\end{aligned}
\end{equation}
where $\eta_{h}(y)=[Ny]/N$, and $[Ny]$ means the integer part of $Ny$. Thus according to \eqref{eqsemischl1Msol1}, we can obtain
\begin{equation}\label{eqsemischl1Msol}
	\begin{aligned}
		u_{h}(x,t)=&\frac{1}{2\pi\mathbf{i}}\int_{\Gamma_{\theta}}e^{zt}z^{\alpha-1}\int_{0}^{1}\tilde{E}_{h}(z,x,y)u_{0}(\eta_{h}(y))dydz\\
		&+\frac{1}{2\pi\mathbf{i}}\int_{\Gamma_{\theta}}e^{zt}\int_{0}^{1}\tilde{E}_{h}(z,x,y)\tilde{f}(\eta_{h}(y))dydz\\
		=&\frac{1}{2\pi\mathbf{i}}\int_{\Gamma_{\theta}}e^{zt}z^{\alpha-1}\int_{0}^{1}\tilde{E}_{h}(z,x,y)u_{0}(\eta_{h}(y))dydz\\
		&+\frac{1}{2\pi\mathbf{i}}\int_{\Gamma_{\theta}}e^{zt}z^{-1}\int_{0}^{1}\tilde{E}_{h}(z,x,y)\tilde{f}_{t}(\eta_{h}(y))dydz\\
		&+\frac{1}{2\pi\mathbf{i}}\int_{\Gamma_{\theta}}e^{zt}z^{-1}\int_{0}^{1}\tilde{E}_{h}(z,x,y)f(\eta_{h}(y),0)dydz,
	\end{aligned}
\end{equation}
where we use $f(t)=f(0)+\int_{0}^{t}f_{t}(s)ds$ and $f_{t}$ means the first derivative of $f$ about t.

In a similar way, the solution of Eq. \eqref{eqretosol} can be represented as
\begin{equation}\label{eqsolofrel}
	\begin{aligned}
		u(x,t)=&\frac{1}{2\pi\mathbf{i}}\int_{\Gamma_{\theta}}e^{zt}z^{\alpha-1}\int_{0}^{1}\tilde{E}(z,x,y)u_{0}(y)dydz\\
		&+\frac{1}{2\pi\mathbf{i}}\int_{\Gamma_{\theta}}e^{zt}\int_{0}^{1}\tilde{E}(z,x,y)\tilde{f}(y)dydz\\
		=&\frac{1}{2\pi\mathbf{i}}\int_{\Gamma_{\theta}}e^{zt}z^{\alpha-1}\int_{0}^{1}\tilde{E}(z,x,y)u_{0}(y)dydz\\
		&+\frac{1}{2\pi\mathbf{i}}\int_{\Gamma_{\theta}}e^{zt}z^{-1}\int_{0}^{1}\tilde{E}(z,x,y)\tilde{f}_{t}(y)dydz\\
		&+\frac{1}{2\pi\mathbf{i}}\int_{\Gamma_{\theta}}e^{zt}z^{-1}\int_{0}^{1}\tilde{E}(z,x,y)f(y,0)dydz,
	\end{aligned}
\end{equation}
where
\begin{equation*}
	\tilde{E}(z,x,y)=\sum_{i=1}^{\infty}(z^{\alpha}+\lambda_{i})^{-1}\phi_{i}(x)\phi_{i}(y).
\end{equation*}

Next, we provide some useful lemmas, which play an important role in the following error analysis.
\begin{lemma}\label{lemspaerr}
	Let $w(x)=\int_{0}^{1}\tilde{E}(z,x,y)v(y)dy$ and $w_{h}(x)=\int_{0}^{1}\tilde{E}_{h}(z,x,y)v(\eta_{h}(y))dy$ with $z\in\Sigma_{\theta}=\{z\in C,|\arg(z)|\leq \theta, |z|>0, \theta \in (\pi/2,\pi) \}$. If $v\in \hat{H}^{\sigma}(\Omega)$ with $\sigma\geq 0$, then it holds
	\begin{equation*}
		\|w-w_{h}\|_{L^{\infty}(\Omega)}\leq Ch^{\min(\sigma+\frac{1}{2}-2\gamma-\epsilon,2)}|z|^{-\gamma\alpha}\|v\|_{\hat{H}^{\sigma}(\Omega)}
	\end{equation*}
	with $\gamma\in[0,\frac{1}{4})$.
\end{lemma}
\begin{proof}
	Write $\ell_{h}=|\ln(h)|^{\frac{1}{2}}$ and let $\chi_{0}(x)$ be a characteristic function on $x=0$, i.e.,
	\begin{equation*}
		\chi_{0}(x)=\left\{\begin{aligned}
			1\quad x=0,\\
			0\quad x\neq 0.
		\end{aligned}\right.
	\end{equation*}
	According to the definitions of $w$ and $w_{h}$, one has
	\begin{equation*}
		\begin{aligned}
			&\|w-w_{h}\|_{L^{\infty}(\Omega)}\\
			\leq&C\left\|\sum_{i=N}^{\infty}(z^{\alpha}+\lambda_{i})^{-1}\phi_{i}(x)(\phi_{i}(y),v(y))\right\|_{L^{\infty}(\Omega)}\\
			&+C\left\|\sum_{i=1}^{N-1}((z^{\alpha}+\lambda_{i})^{-1}-(z^{\alpha}+\lambda_{h,i})^{-1})\phi_{i}(x)(\phi_{i}(y),v(y))\right\|_{L^{\infty}(\Omega)}\\
			&+C\left\|\sum_{i=1}^{N-1}(z^{\alpha}+\lambda_{h,i})^{-1}(\phi_{i}(x)-\phi_{h,i}(x))(\phi_{i}(y),v(y))\right\|_{L^{\infty}(\Omega)}\\
			&+C\left\|\sum_{i=1}^{N-1}(z^{\alpha}+\lambda_{h,i})^{-1}\phi_{h,i}(x)((\phi_{i}(y),v(y))-(\phi_{i}(\eta_{h}(y)),v(\eta_{h}(y))))\right\|_{L^{\infty}(\Omega)}\\
			\leq& \uppercase\expandafter{\romannumeral1}+\uppercase\expandafter{\romannumeral2}+\uppercase\expandafter{\romannumeral3}+\uppercase\expandafter{\romannumeral4}.
		\end{aligned}
	\end{equation*}
	By using $v\in \hat{H}^{\sigma}(\Omega)$, $\lambda_{k}=k^{2}\pi^{2}$,  $\|\phi_{i}(x)\|_{L^{\infty}(\Omega)}\leq C$, and $|\lambda^{1-\gamma}_{i}(z^{\alpha}+\lambda_{i})^{-1}|\leq C|z|^{-\gamma\alpha}$ for $z\in\Sigma_{\theta}$ and $\gamma\in[0,1]$ \cite{Lubich.1996Ndeefaoaeewapmt}, the first term $\uppercase\expandafter{\romannumeral1}$ can be bounded by
	\begin{equation*}
		\begin{aligned}
			\uppercase\expandafter{\romannumeral1}\leq&C\left\|\sum_{i=N}^{\infty}\lambda^{\gamma-1}_{i}\lambda_{i}^{1-\gamma}(z^{\alpha}+\lambda_{i})^{-1}\phi_{i}(x)(\phi_{i}(y),v(y))\right\|_{L^{\infty}(\Omega)}\\
			\leq& C\lambda_{N}^{\frac{1}{4}-\frac{\sigma}{2}+\gamma-1+\epsilon}\sum_{i=N}^{\infty}|\lambda_{i}^{1-\gamma}(z^{\alpha}+\lambda_{i})^{-1}|\|\phi_{i}(x)\|_{L^{\infty}(\Omega)}|\lambda^{\frac{\sigma}{2}-\frac{1}{4}-\epsilon}_{i}(\phi_{i}(y),v(y))|\\
			\leq& C\lambda_{N}^{\frac{1}{4}-\frac{\sigma}{2}+\gamma-1+\epsilon}\left(\sum_{i=N}^{\infty}|z|^{-2\gamma\alpha}\lambda^{\sigma}_{i}(\phi_{i}(y),v(y))^{2}\right)^{\frac{1}{2}}\left (\sum_{i=N}^{\infty}\lambda_{i}^{-\frac{1}{2}-2\epsilon}\right)^{\frac{1}{2}}\\
			\leq& C\lambda_{N}^{\frac{1}{4}-\frac{\sigma}{2}+\gamma-1}|z|^{-\gamma\alpha}\|v\|_{\hat{H}^{\sigma}(\Omega)}.
		\end{aligned}
	\end{equation*}
	To estimate $\uppercase\expandafter{\romannumeral2}$, we introduce $K_{j}=1-(c^{N}_{j})^{-1}$. Then
	\begin{equation*}
		\begin{aligned}
			|K_{j}|=\frac{\left(\frac{j\pi}{2N}\right)^{2}}{\sin^{2}\left(\frac{j\pi}{2N}\right)}-1
			=\frac{\left(\frac{j\pi}{2N}\right)^{2}-\sin^{2}\left(\frac{j\pi}{2N}\right )}{\sin^{2}\left(\frac{j\pi}{2N}\right)}.
		\end{aligned}
	\end{equation*}
	Since
	\begin{equation*}
		\lim_{x\rightarrow 0}\frac{\frac{x^{2}-\sin^{2}(x)}{\sin^{2}(x)}}{x^{2}}=\lim_{x\rightarrow 0}\frac{x^{2}-\sin^{2}(x)}{\sin^{2}(x)x^{2}}=\frac{1}{3},
	\end{equation*}
	it follows that
	\begin{equation*}
		|K_{j}|\leq C\left (\frac{j\pi}{2N}\right )^{2}.
	\end{equation*}
	%	for any $\epsilon_{0}\in(0,\frac{1}{6})$, then there exists $\delta_{0}>0$ such that
	%	\begin{equation*}
		%		\left |\frac{K_{j}}{\left(\frac{j\pi}{2N}\right)^{2}}-\frac{1}{3}\right |\leq \epsilon_{0}\leq \frac{1}{6}\quad {\rm for}~~\frac{j\pi}{2N}<\delta_{0},
		%	\end{equation*}
	%	which leads to
	%	\begin{equation*}
		%		K_{j}\leq C\left (\frac{j\pi}{2N}\right )^{2}\quad {\rm for}~~\frac{j\pi}{2N}<\delta_{0}.
		%	\end{equation*}
	%	As for $\frac{j\pi}{2N}>\delta_{0}$, it's easy to see that
	%	\begin{equation*}
		%		K_{j}\leq C\left (\frac{j\pi}{2N}\right )^{2}.
		%	\end{equation*}
	Thus combining the following facts
	\begin{equation*}
		\begin{aligned}
			&(z^{\alpha}+\lambda_{i})^{-1}-(z^{\alpha}+\lambda_{h,i})^{-1}\\
			=&(\lambda_{h,i}-\lambda_{i})(z^{\alpha}+\lambda_{i})^{-1}(z^{\alpha}+\lambda_{h,i})^{-1}\\
			=&K_{i}\lambda_{h,i}(z^{\alpha}+\lambda_{i})^{-1}(z^{\alpha}+\lambda_{h,i})^{-1},
		\end{aligned}
	\end{equation*}
	$\lambda_{k}=k^{2}\pi^{2}$,	and $|\lambda^{1-\gamma}_{i}(z^{\alpha}+\lambda_{i})^{-1}|\leq C|z|^{-\gamma\alpha}$ for $z\in\Sigma_{\theta}$ and $\gamma\in[0,1]$ \cite{Lubich.1996Ndeefaoaeewapmt}, we obtain
	\begin{equation*}
		\begin{aligned}
			\uppercase\expandafter{\romannumeral2}\leq&C\left\|\sum_{i=1}^{N-1}(1-(c^{N}_{i})^{-1})\lambda_{h,i}(z^{\alpha}+\lambda_{i})^{-1}(z^{\alpha}+\lambda_{h,i})^{-1}\phi_{i}(x)(\phi_{i}(y),v(y))\right\|_{L^{\infty}(\Omega)}\\
			\leq& C\lambda_{N}^{-1}\left(\sum_{i=1}^{N-1}|\lambda_{i}^{1-\gamma}(z^{\alpha}+\lambda_{i})^{-1}|^{2}\lambda^{\sigma}_{i}(\phi_{i}(y),v(y))^{2}\right)^{\frac{1}{2}}\left(\sum_{i=1}^{N-1}\lambda^{2\gamma-\sigma}_{i}\right)^{\frac{1}{2}}\\
			\leq& C\lambda_{N}^{\max(\frac{1}{4}-\frac{\sigma}{2}+\gamma-1,-1)}\ell_{h}^{\chi_{0}(\frac{1}{4}-\frac{\sigma}{2}+\gamma)}|z|^{-\gamma\alpha}\|v\|_{\hat{H}^{\sigma}(\Omega)}.
		\end{aligned}
	\end{equation*}
	Similarly, using the interpolation theory \cite{Brenner.2008TMToFEM}, one has
	\begin{equation*}
		\begin{aligned}
			\uppercase\expandafter{\romannumeral3}\leq&  C\left(\sum_{i=1}^{N-1}(\lambda_{i}^{1-\gamma}(z^{\alpha}+\lambda_{i})^{-1})^{2}\lambda^{\sigma}_{i}(\phi_{i}(y),v(y))^{2}\right)^{\frac{1}{2}}\left \|\sum_{i=1}^{N-1}\lambda_{i}^{2\gamma-2-\sigma}(\phi_{i}(x)-\phi_{h,i}(x))^{2}\right\|_{L^{\infty}(\Omega)}^{\frac{1}{2}}\\
			\leq& C\lambda_{N}^{-1}\left(\sum_{i=1}^{N-1}(\lambda_{i}^{1-\gamma}(z^{\alpha}+\lambda_{i})^{-1})^{2}\lambda^{\sigma}_{i}(\phi_{i}(y),v(y))^{2}\right )^{\frac{1}{2}}\left (\sum_{i=1}^{N-1}\lambda_{i}^{2\gamma-\sigma}\right)^{\frac{1}{2}}\\
			\leq&C\lambda_{N}^{\max(\frac{1}{4}-\frac{\sigma}{2}+\gamma-1,-1)}\ell_{h}^{\chi_{0}(\frac{1}{4 }-\frac{\sigma}{2}+\gamma)}|z|^{-\gamma\alpha}\|v\|_{\hat{H}^{\sigma}(\Omega)}.
		\end{aligned}
	\end{equation*}
	As for $	\uppercase\expandafter{\romannumeral4}$, splitting $v$ into two parts, i.e.,
	\begin{equation*}
		\begin{aligned}
			v=v_{N}+v_{R}=\sum_{k=1}^{N-1}(v,\phi_{k})\phi_{k}+\sum_{k=N}^{\infty}(v,\phi_{k})\phi_{k},
		\end{aligned}
	\end{equation*}
	and using discrete H\"{o}lder's inequality, we have
	\begin{equation*}
		\begin{aligned}
			\uppercase\expandafter{\romannumeral4}\leq&C\left(\sum_{i=1}^{N-1}\lambda_{h,i}^{-1/2-\epsilon}|\lambda_{h,i}^{1-\gamma}(z^{\alpha}+\lambda_{h,i})^{-1}|^{2}\|\phi_{h,i}(x)\|_{L^{\infty}(\Omega)}^{2}\right)^{\frac{1}{2}}\\
			&\cdot\Bigg(\sum_{i=1}^{N-1}\lambda_{h,i}^{\epsilon-\frac{3}{2}+2\gamma}((\phi_{i}(y),v(y))-(\phi_{i}(\eta_{h}(y)),v_{N}(\eta_{h}(y))))^{2}\\
			&\quad+\sum_{i=1}^{N-1}\lambda_{h,i}^{\epsilon-\frac{3}{2}+2\gamma}((\phi_{i}(\eta_{h}(y)),v_{R}(\eta_{h}(y))))^{2}\Bigg)^{\frac{1}{2}}.
		\end{aligned}
	\end{equation*}
	By using the definition of $v_{N}$ and the orthogonality of $\{\boldsymbol{\varphi}_{h,k}\}_{k=1}^{\infty}$, one can get
	\begin{equation*}
		\begin{aligned}
			\sum_{i=1}^{N-1}\lambda_{h,i}^{\epsilon-\frac{3}{2}+2\gamma}((\phi_{i}(y),v(y))-(\phi_{i}(\eta_{h}(y)),v_{N}(\eta_{h}(y))))^{2}=0.
		\end{aligned}
	\end{equation*}
	As for $\sum_{i=1}^{N-1}\lambda_{h,i}^{\epsilon-\frac{3}{2}+2\gamma}((\phi_{i}(\eta_{h}(y)),v_{R}(\eta_{h}(y))))^{2}$, simple calculations result in
	\begin{equation*}
		\begin{aligned}
			&(\phi_{i}(\eta_{h}(y)),v_{R}(\eta_{h}(y)))
			\leq\|\phi_{i}(\eta_{h}(y))\|_{\hat{H}^{\frac{1}{2}-2\gamma-\epsilon}(\Omega)}\|v_{R}(\eta_{h}(y))\|_{\hat{H}^{-\frac{1}{2}+2\gamma+\epsilon}(\Omega)},
		\end{aligned}
	\end{equation*}
	where we need to require $\gamma\in[0,\frac{1}{4})$. Using the definition of $v_{R}$ and the interpolation theorem \cite{Brenner.2008TMToFEM} leads to
	\begin{equation*}
		\begin{aligned}
			&\|\phi_{i}(\eta_{h}(y))\|_{\hat{H}^{\frac{1}{2}-2\gamma-\epsilon}(\Omega)}\\
			\leq&\|\phi_{i}(\eta_{h}(y))-\phi_{i}(y)\|_{\hat{H}^{\frac{1}{2}-2\gamma-\epsilon}(\Omega)}+\|\phi_{i}(y)\|_{\hat{H}^{\frac{1}{2}-2\gamma-\epsilon}(\Omega)}\\
			\leq& C\|\phi_{i}\|_{\hat{H}^{\frac{1}{2}-2\gamma-\epsilon}(\Omega)}
		\end{aligned}
	\end{equation*}
	and
	\begin{equation*}
		\begin{aligned}
			&\|v_{R}(\eta_{h}(y))\|_{\hat{H}^{-\frac{1}{2}+2\gamma+\epsilon}(\Omega)}\\
			\leq&C\left(\sum_{k=N}^{\infty}\lambda_{k}^{-\frac{1}{2}+2\gamma+\epsilon}(v_{R}(\eta_{h}(y)),\phi_{k}(y))^{2}\right)^{1/2}\\
			\leq&C\lambda_{N}^{-\frac{1}{4}+\gamma+\epsilon/2-\frac{\sigma}{2}}\|v_{R}(\eta_{h}(y))\|_{\hat{H}^{\sigma}(\Omega)}\\
			\leq& C\lambda_{N}^{-\frac{1}{4}+\gamma+\epsilon/2-\frac{\sigma}{2}}\|v\|_{\hat{H}^{\sigma}(\Omega)}.
		\end{aligned}
	\end{equation*}
	Thus
	\begin{equation*}
		\begin{aligned}
			&\left (\sum_{i=1}^{N-1}\lambda_{h,i}^{\epsilon-\frac{3}{2}+2\gamma}((\phi_{i}(\eta_{h}(y)),v_{R}(\eta_{h}(y))))^{2}\right )^{\frac{1}{2}}\\
			\leq&C\lambda_{N}^{-\frac{1}{4}-\frac{\sigma}{2}+\gamma+\epsilon}\left (\sum_{i=1}^{N-1}\lambda_{h,i}^{-1+\epsilon}\right) ^{\frac{1}{2}}\|v\|_{\hat{H}^{\sigma}(\Omega)}
			\leq C\lambda_{N}^{-\frac{1}{4}-\frac{\sigma}{2}+\gamma+\epsilon}\|v\|_{\hat{H}^{\sigma}(\Omega)},
		\end{aligned}
	\end{equation*}
	which leads to
	\begin{equation*}
		\uppercase\expandafter{\romannumeral4}\leq C\lambda_{N}^{-\frac{1}{4}-\frac{\sigma}{2}+\gamma+\epsilon}|z|^{-\gamma\alpha}\|v\|_{\hat{H}^{\sigma}(\Omega)}.
	\end{equation*}
	Therefore, the desired result follows after collecting the above estimates.
\end{proof}

\begin{remark}\label{Reignore3}
	In fact, by the definition of $\phi_{h,i}$, the estimate about $\uppercase\expandafter{\romannumeral3}$ in the proof of Lemma \ref{lemspaerr} can be ignored if we just focus on the errors on $\{x_{i}\}_{i=1}^{N-1}$.
\end{remark}

\begin{lemma}\label{lemspaerrl2}
	Let $w(x)=\int_{0}^{1}\tilde{E}(z,x,y)v(y)dy$ and $w_{h}(x)=\int_{0}^{1}\tilde{E}_{h}(z,x,y)v(\eta_{h}(y))dy$ with $z\in\Sigma_{\theta}=\{z\in C,|\arg(z)|\leq \theta, |z|>0, \theta \in (\pi/2,\pi) \}$. If $v\in \hat{H}^{\sigma}(\Omega)$  with $\sigma\geq 0$, then it holds
	\begin{equation*}
		\|\mathcal{I}_{N}w-\mathcal{I}_{N}w_{h}\|_{l^{2}}\leq Ch^{\min(\sigma+\frac{1}{2}-2\gamma-\epsilon,2)}|z|^{-\gamma\alpha}\|v\|_{\hat{H}^{\sigma}(\Omega)}
	\end{equation*}
	with $\gamma\in[0,\frac{1}{4})$ and $\mathcal{I}_{N}$ being defined in \eqref{eqdefI}.
\end{lemma}
\begin{proof}
	According to the definitions of $w$ and $w_{h}$, one has
	\begin{equation*}
		\begin{aligned}
			&\|\mathcal{I}_{N}w-\mathcal{I}_{N}w_{h}\|_{l^{2}}\\
			\leq&C\left\|\mathcal{I}_{N}\sum_{i=N}^{\infty}(z^{\alpha}+\lambda_{i})^{-1}\phi_{i}(x)(\phi_{i}(y),v(y))\right\|_{l^{2}}\\
			&+C\left\|\mathcal{I}_{N}\sum_{i=1}^{N-1}((z^{\alpha}+\lambda_{i})^{-1}-(z^{\alpha}+\lambda_{h,i})^{-1})\phi_{i}(x)(\phi_{i}(y),v(y))\right\|_{l^{2}}\\
			&+C\left\|\mathcal{I}_{N}\sum_{i=1}^{N-1}(z^{\alpha}+\lambda_{h,i})^{-1}(\phi_{i}(x)-\phi_{h,i}(x))(\phi_{i}(y),v(y))\right\|_{l^{2}}\\
			&+C\left\|\mathcal{I}_{N}\sum_{i=1}^{N-1}(z^{\alpha}+\lambda_{h,i})^{-1}\phi_{h,i}(x)((\phi_{i}(y),v(y))-(\phi_{i}(\eta_{h}(y)),v(\eta_{h}(y))))\right\|_{l^{2}}\\
			\leq& \uppercase\expandafter{\romannumeral1}+\uppercase\expandafter{\romannumeral2}+\uppercase\expandafter{\romannumeral3}+\uppercase\expandafter{\romannumeral4}.
		\end{aligned}
	\end{equation*}
	By using $v\in \hat{H}^{\sigma}(\Omega)$, $\lambda_{k}=k^{2}\pi^{2}$,  $\|\phi_{i}(x)\|_{L^{\infty}(\Omega)}\leq C$, and $|\lambda^{1-\gamma}_{i}(z^{\alpha}+\lambda_{i})^{-1}|\leq C|z|^{-\gamma\alpha}$ for $z\in\Sigma_{\theta}$ and $\gamma\in[0,1]$ \cite{Lubich.1996Ndeefaoaeewapmt}, $\uppercase\expandafter{\romannumeral1}$ can be bounded by
	\begin{equation*}
		\begin{aligned}
			\uppercase\expandafter{\romannumeral1}^{2}\leq&C\sum_{i=N}^{\infty}(\lambda^{\gamma-1}_{i}\lambda_{i}^{1-\gamma}(z^{\alpha}+\lambda_{i})^{-1})^{2}(\phi_{i}(y),v(y))^{2}\\
			\leq& C\lambda_{N}^{-\sigma+2\gamma-2}\sum_{i=N}^{\infty}(\lambda_{i}^{1-\gamma}(z^{\alpha}+\lambda_{i})^{-1})^{2}\lambda^{\sigma}_{i}(\phi_{i}(y),v(y))^{2}\\
			\leq& C\lambda_{N}^{-\sigma+2\gamma-2}|z|^{-2\gamma\alpha}\|v\|_{\hat{H}^{\sigma}(\Omega)}^{2}.
		\end{aligned}
	\end{equation*}
	According to the proof of Lemma \ref{lemspaerr}, one has $|K_{j}|=|1-(c^{N}_{j})^{-1}|\leq C\left (\frac{j\pi}{2N}\right )^{2}$, which leads to
	\begin{equation*}
		\begin{aligned}
			\uppercase\expandafter{\romannumeral2}^{2}\leq&C\sum_{i=1}^{N-1}(K_{i}\lambda_{h,i}(z^{\alpha}+\lambda_{i})^{-1}(z^{\alpha}+\lambda_{h,i})^{-1})^{2}(\phi_{i}(y),v(y))^{2}\\
			\leq& C\lambda_{N}^{-2+\max(2\gamma-\sigma,0)}\sum_{i=1}^{N-1}\left (\lambda_{i}^{1-\gamma}(z^{\alpha}+\lambda_{i})^{-1}\right)^{2}\lambda_{i}^{\sigma}(\phi_{i}(y),v(y) )^{2}\\
			\leq& C\lambda_{N}^{-2+\max(2\gamma-\sigma,0)}|z|^{-2\gamma\alpha}\|v\|^{2}_{\hat{H}^{\sigma}(\Omega)}.
		\end{aligned}
	\end{equation*}
	From Remark \ref{Reignore3}, one can easily get
	\begin{equation*}
		\begin{aligned}
			\uppercase\expandafter{\romannumeral3}=0.
		\end{aligned}
	\end{equation*}
	As for $	\uppercase\expandafter{\romannumeral4}$, splitting $v$ into two parts, i.e.,
	\begin{equation*}
		\begin{aligned}
			v=v_{N}+v_{R}=\sum_{k=1}^{N-1}(v,\phi_{k})\phi_{k}+\sum_{k=N}^{\infty}(v,\phi_{k})\phi_{k},
		\end{aligned}
	\end{equation*}
	and using the definition of $v_{n}$ and orthogonality of $\{\boldsymbol{\varphi}_{h,k}\}_{k=1}^{\infty}$, we have
	\begin{equation*}
		\begin{aligned}
			\uppercase\expandafter{\romannumeral4}^{2}\leq&C\sum_{i=1}^{N-1}\left((z^{\alpha}+\lambda_{h,i})^{-1}\right)^{2}(\phi_{i}(\eta_{h}(y)),v_{R}(\eta_{h}(y)))^{2}.
		\end{aligned}
	\end{equation*}
	From the proof of Lemma \ref{lemspaerr}, one has
	\begin{equation*}
		\begin{aligned}
			(\phi_{i}(\eta_{h}(y)),v_{R}(\eta_{h}(y)))
			\leq C\lambda_{N}^{-\frac{\sigma-\epsilon}{2}-\frac{1}{4}}\|\phi_{i}(\eta_{h}(y))\|_{\hat{H}^{\frac{1}{2}-\epsilon}(\Omega)}\|v\|_{\hat{H}^{\sigma}(\Omega)}.
		\end{aligned}
	\end{equation*}
	Thus
	\begin{equation*}
		\uppercase\expandafter{\romannumeral4}\leq C\lambda_{N}^{-\frac{1}{4}-\frac{\sigma}{2}+\gamma+\epsilon}|z|^{-\gamma\alpha}\|v\|_{\hat{H}^{\sigma}(\Omega)}.
	\end{equation*}
	Therefore, the desired result follows after collecting the above estimates.
\end{proof}

Now we turn to the spatial error estimates  for solving homogeneous and inhomogeneous problems \eqref{eqretosol}, respectively.
\begin{theorem}\label{thmerrorhom}
	Let $u$ and $\mathbf{u}_{h}$, defined in \eqref{eqdefbu} and \eqref{eqdefuh}, be the solutions of Eqs. \eqref{eqretosol} and \eqref{eqsemischl1}, respectively.  Let $u_{0}\in \hat{H}^{\sigma}(\Omega)$ and $f=0$ with $\sigma\geq 0$. Then one has
	\begin{equation*}
		\begin{aligned}
			\|\mathbf{u}-\mathbf{u}_{h}\|_{l^{\infty}}\leq& \|u-u_{h}\|_{L^{\infty}(\Omega)}
			\leq Ch^{\min(\sigma+\frac{1}{2}-2\gamma-\epsilon,2)}t^{-\alpha}\|u_{0}\|_{\hat{H}^{\sigma}(\Omega)}
		\end{aligned}
	\end{equation*}
	and
	\begin{equation*}
		\begin{aligned}
			\|\mathbf{u}-\mathbf{u}_{h}\|_{l^{2}}
			\leq Ch^{\min(\sigma+\frac{1}{2}-2\gamma-\epsilon,2)}t^{-\alpha}\|u_{0}\|_{\hat{H}^{\sigma}(\Omega)}.
		\end{aligned}
	\end{equation*}
\end{theorem}
\begin{proof}
	According to \eqref{eqsolofrel} and \eqref{eqsemischl1Msol}, we have
	\begin{equation*}
		\begin{aligned}
			&\|u-u_{h}\|_{L^{\infty}(\Omega)}
			\leq C\left \|\int_{\Gamma_{\theta}}e^{zt}z^{\alpha-1}\left (\int_{0}^{1}\tilde{E}(z,x,y)u_{0}(y)dy-\int_{0}^{1}\tilde{E}_{h}(z,x,y)u_{0}(\eta_{h}(y))dy\right )dz\right \|_{L^{\infty}(\Omega)}.
		\end{aligned}
	\end{equation*}
	Lemma \ref{lemspaerr} and simple calculations imply
	\begin{equation*}
		\begin{aligned}
			\|u-u_{h}\|_{L^{\infty}(\Omega)}\leq&C\int_{\Gamma_{\theta}}|e^{zt}||z|^{\alpha-1}\left \|\int_{0}^{1}\tilde{E}(z,x,y)u_{0}(y)dy-\int_{0}^{1}\tilde{E}_{h}(z,x,y)u_{0}(\eta_{h}(y))dy\right \|_{L^{\infty}(\Omega)}|dz|\\
			\leq&Ch^{\min(\sigma+\frac{1}{2}-2\gamma-\epsilon,2)}\int_{\Gamma_{\theta}}|e^{zt}||z|^{\alpha-1}|dz|\|u_{0}\|_{\hat{H}^{\sigma}(\Omega)}\\
			\leq&Ch^{\min(\sigma+\frac{1}{2}-2\gamma-\epsilon,2)}t^{-\alpha}\|u_{0}\|_{\hat{H}^{\sigma}(\Omega)},
		\end{aligned}
	\end{equation*}
	from which  the first desired result follows. As for the second estimate, one can obtain similarly.
\end{proof}

Similar to the proof of Theorem \ref{thmerrorhom},  the following error estimate for inhomogeneous problem \eqref{eqretosol} can be got.
\begin{theorem}\label{thmerrorimhom}
	Let $u$ and $\mathbf{u}_{h}$, defined in \eqref{eqdefbu} and \eqref{eqdefuh}, be the solutions of Eqs. \eqref{eqretosol} and \eqref{eqsemischl1}. Assume $u_{0}=0$, $f(0)\in \hat{H}^{\sigma}(\Omega)$, and $\int_{0}^{t}\|f_{t}(s)\|_{\hat{H}^{\sigma}(\Omega)}ds<\infty$ with $\sigma\geq 0$. Then one has
	\begin{equation*}
		\begin{aligned}
			\|\mathbf{u}-\mathbf{u}_{h}\|_{l^{\infty}}\leq& \|u-u_{h}\|_{L^{\infty}(\Omega)}\\
			\leq& Ch^{\min(\sigma+\frac{1}{2}-\epsilon,2)}\left (\|f(0)\|_{\hat{H}^{\sigma}(\Omega)}+\int_{0}^{t}\|f_{t}(s)\|_{\hat{H}^{\sigma}(\Omega)}ds\right ),
		\end{aligned}
	\end{equation*}
	and
	\begin{equation*}
		\begin{aligned}
			\|\mathbf{u}-\mathbf{u}_{h}\|_{l^{2}}\leq Ch^{\min(\sigma+\frac{1}{2}-\epsilon,2)}\left (\|f(0)\|_{\hat{H}^{\sigma}(\Omega)}+\int_{0}^{t}\|f_{t}(s)\|_{\hat{H}^{\sigma}(\Omega)}ds\right ).
		\end{aligned}
	\end{equation*}
\end{theorem}
\begin{remark}
	In fact, if $u_{0}=0$ and $\int_{0}^{t}(t-s)^{-1+\epsilon\alpha}\|f(s)\|_{\hat{H}^{\sigma}(\Omega)}ds<\infty$ with $\sigma\in[0,\frac{3}{2}]$, the corresponding spatial error estimates can be written as
	\begin{equation*}
		\begin{aligned}
			\|\mathbf{u}-\mathbf{u}_{h}\|_{l^{\infty}}\leq& \|u-u_{h}\|_{L^{\infty}(\Omega)}
			\leq Ch^{\sigma+\frac{1}{2}-3\epsilon}\int_{0}^{t}(t-s)^{-1+\epsilon\alpha}\|f(s)\|_{\hat{H}^{\sigma}(\Omega)}ds
		\end{aligned}
	\end{equation*}
	and
	\begin{equation*}
		\begin{aligned}
			\|\mathbf{u}-\mathbf{u}_{h}\|_{l^{2}}
			\leq Ch^{\sigma+\frac{1}{2}-3\epsilon}\int_{0}^{t}(t-s)^{-1+\epsilon\alpha}\|f(s)\|_{\hat{H}^{\sigma}(\Omega)}ds,
		\end{aligned}
	\end{equation*}
	which can be got by Lemmas \ref{lemspaerr} and \ref{lemspaerrl2}. While if $u_{0}=0$ and $\int_{0}^{t}(t-s)^{-1+\gamma\alpha}\|f(s)\|_{\hat{H}^{\sigma}(\Omega)}ds<\infty$ with $\sigma>\frac{3}{2}$ and $\gamma=\min(\frac{2\sigma-3}{8},\frac{1}{8})>0$, the spatial error estimates become
	\begin{equation*}
		\begin{aligned}
			\|\mathbf{u}-\mathbf{u}_{h}\|_{l^{\infty}}\leq& \|u-u_{h}\|_{L^{\infty}(\Omega)}
			\leq Ch^{2}\int_{0}^{t}(t-s)^{-1+\gamma\alpha}\|f(s)\|_{\hat{H}^{\sigma}(\Omega)}ds
		\end{aligned}
	\end{equation*}
	and
	\begin{equation*}
		\begin{aligned}
			\|\mathbf{u}-\mathbf{u}_{h}\|_{l^{2}}
			\leq Ch^{2}\int_{0}^{t}(t-s)^{-1+\gamma\alpha}\|f(s)\|_{\hat{H}^{\sigma}(\Omega)}ds.
		\end{aligned}
	\end{equation*}
\end{remark}

\subsection{The modified central finite difference scheme and error analysis}
Here we first introduce the following projection, i.e., define the projection operator $P_{N}:L^{2}(\Omega)\rightarrow \mathbb{H}_{N}$ satisfying
\begin{equation*}
	(P_{N}u,v_{N})=(u,v_{N})\quad \forall v\in \mathbb{H}_{N},
\end{equation*}
where $\mathbb{H}_{N}={\rm span}\{\phi_{k}\}_{k=1}^{N-1}$. It is easy to see  that
\begin{equation*}
	P_{N}u=\sum_{k=1}^{N-1}(u,\phi_{k})\phi_{k}.
\end{equation*}

From the proof of Lemma \ref{lemspaerr}, it can be noted that the spatial convergence rates are mainly limited by the estimate of $(\phi_{i}(\eta_{h}(y)),v_{R}(\eta_{h}(y)))$ with $v_{R}=\sum_{k=N}^{\infty}(v,\phi_{k})\phi_{k}$. A direct idea about improving  the convergence rates is how to avoid these terms in error analysis. Based on this idea, a modified central finite difference scheme can be provided, i.e.,
\begin{equation}\label{eqsemischmod}
	\begin{aligned}
		&{}_{0}\partial^{\alpha}_{t}(u_{h,p,i}-u^{0}_{p,i})+N^{2}\sum_{j=1}^{N-1}w_{i,j}u_{h,p,j}=f_{p,i}.
	\end{aligned}
\end{equation}
Here $u_{h,p,i}$ is the numerical solution of $u$ at $x_{i}$, $u^{0}_{p,i}=(P_{N}u_{0})(x_{i})$, and $f_{p,i}=(P_{N}f)(x_{i})$.

Denote $\mathbf{u}_{h,p}=[u_{h,p,1},u_{h,p,2},\ldots,u_{h,p,N-1}]^{T}$, $\mathbf{u}^{0}_{p}=[u^{0}_{p,1},u^{0}_{p,2},\ldots,u^{0}_{p,N-1}]^{T}=\mathcal{I}_{N}(P_{N}u_{0})=\mathcal{I}_{N}u_{0,N}$, and $\mathbf{f}_{p}=[f_{p,1},f_{p,2},\ldots,f_{p,N-1}]^{T}=\mathcal{I}_{N}(P_{N}f)=\mathcal{I}_{N}f_{N}$. Thus the solution of Eq. \eqref{eqsemischmod} is
\begin{equation}\label{eqsemischl1Msol02}
	\begin{aligned}
		\mathbf{u}_{h,p}(t)=\frac{1}{2\pi\mathbf{i}}\int_{\Gamma_{\theta}}e^{zt}z^{\alpha-1}(z^{\alpha}+\mathbf{A})^{-1}\mathbf{u}^{0}_{p}dz+\frac{1}{2\pi\mathbf{i}}\int_{\Gamma_{\theta}}e^{zt}(z^{\alpha}+\mathbf{A})^{-1}\tilde{\mathbf{f}}_{p}dz.
	\end{aligned}
\end{equation}
Similarly, with $u_{h,p,0}=u_{h,p,N}=0$, introduce
\begin{equation}\label{eqdefuhp}
	u_{h,p}(x,t)=u_{h,p,i}+(Nx-i)(u_{h,p,i+1}-u_{h,p,i})\qquad {\rm} x\in[x_{i},x_{i+1}),~i=0,1,\ldots,N-1.
\end{equation}
Then $u_{h,p}$ can be represented as
\begin{equation}\label{eqsemischl2Msol}
	\begin{aligned}
		u_{h,p}(x,t)=&\frac{1}{2\pi\mathbf{i}}\int_{\Gamma_{\theta}}e^{zt}z^{\alpha-1}\int_{0}^{1}\tilde{E}_{h}(z,x,y)u_{0,N}(\eta_{h}(y))dydz\\
		&+\frac{1}{2\pi\mathbf{i}}\int_{\Gamma_{\theta}}e^{zt}\int_{0}^{1}\tilde{E}_{h}(z,x,y)\tilde{f}_{N}(\eta_{h}(y))dydz\\
		=&\frac{1}{2\pi\mathbf{i}}\int_{\Gamma_{\theta}}e^{zt}z^{\alpha-1}\int_{0}^{1}\tilde{E}_{h}(z,x,y)u_{0,N}(\eta_{h}(y))dydz\\
		&+\frac{1}{2\pi\mathbf{i}}\int_{\Gamma_{\theta}}e^{zt}z^{-1}\int_{0}^{1}\tilde{E}_{h}(z,x,y)\tilde{f}_{N,t}(\eta_{h}(y))dydz\\
		&+\frac{1}{2\pi\mathbf{i}}\int_{\Gamma_{\theta}}e^{zt}z^{-1}\int_{0}^{1}\tilde{E}_{h}(z,x,y)f_{N}(\eta_{h}(y),0)dydz
	\end{aligned}
\end{equation}
with $f_{N,t}$ being the first derivative of $f_{N}$ about $t$.

Following the proofs of Lemmas \ref{lemspaerr}, \ref{lemspaerrl2}, and Remark \ref{Reignore3}, one can easily get the following lemma.
\begin{lemma}\label{lemspaerr2}
	Let $w(x)=\int_{0}^{1}\tilde{E}(z,x,y)v(y)dy$ and $w_{h}(x)=\int_{0}^{1}\tilde{E}_{h}(z,x,y)v_{N}(\eta_{h}(y))dy$ with $z\in\Sigma_{\theta}$ and $v_{N}=P_{N}v$. If $v\in \hat{H}^{\sigma}(\Omega)$ with $\sigma\geq 0$, then it holds
	\begin{equation*}
		\|\mathcal{I}_{N}w-\mathcal{I}_{N}w_{h}\|_{l^{\infty}}\leq Ch^{\max(\sigma+\frac{3}{2}-2\gamma,2)}\ell_{h}^{\chi_{0}(\frac{1}{4 }-\frac{\sigma}{2}+\gamma)}|z|^{-\gamma\alpha}\|v\|_{\hat{H}^{\sigma}(\Omega)}
	\end{equation*}
	and
	\begin{equation*}
		\|\mathcal{I}_{N}w-\mathcal{I}_{N}w_{h}\|_{l^{2}}\leq Ch^{2-\max(2\gamma-\sigma,0)}|z|^{-\gamma\alpha}\|v\|_{\hat{H}^{\sigma}(\Omega)}
	\end{equation*}
	with $\gamma\in[0,\frac{1}{4})$, $\ell_{h}=|\ln(h)|^{\frac{1}{2}}$, and $\chi_{0}(x)$ is the characteristic function on $x=0$.
\end{lemma}

Similarly, one can obtain the error estimates of the modified scheme \eqref{eqsemischmod} for solving the homogeneous and inhomogeneous problems \eqref{eqretosol}, respectively.
\begin{theorem}\label{thmerrorhom2}
	Let $u$ and $\mathbf{u}_{h,p}$, defined in \eqref{eqdefbu} and \eqref{eqdefuhp}, be the solutions of Eqs. \eqref{eqretosol} and \eqref{eqsemischmod}, respectively. %And $\mathbf{u}$ and $u_{h,p}$ are defined in \eqref{eqdefbu} and \eqref{eqdefuhp}, respectively.
	Let $u_{0}\in \hat{H}^{\sigma}(\Omega)$ and $f=0$ with $\sigma\geq 0$. Then one has
	\begin{equation*}
		\begin{aligned}
			\|\mathbf{u}-\mathbf{u}_{h,p}\|_{l^{\infty}}\leq Ch^{\min(\sigma+\frac{3}{2},2)}\ell_{h}^{\chi_{0}(\frac{1}{4 }-\frac{\sigma}{2})}t^{-\alpha}\|u_{0}\|_{\hat{H}^{\sigma}(\Omega)}
		\end{aligned}
	\end{equation*}
	and
	\begin{equation*}
		\|\mathbf{u}-\mathbf{u}_{h,p}\|_{l^{2}}\leq Ch^{2}t^{-\alpha}\|u_{0}\|_{\hat{H}^{\sigma}(\Omega)}.
	\end{equation*}
	
\end{theorem}

\begin{theorem}\label{thmerrorinhom2}
	Let $u$ and $\mathbf{u}_{h,p}$, defined in \eqref{eqdefbu} and \eqref{eqdefuhp}, be the solutions of Eqs. \eqref{eqretosol} and \eqref{eqsemischmod}, respectively. %And $\mathbf{u}$ and $u_{h,p}$ are defined in \eqref{eqdefbu} and \eqref{eqdefuhp}, respectively.
	Assume $u_{0}=0$, $f(0)\in \hat{H}^{\sigma}(\Omega)$, and $\int_{0}^{t}\|f_{t}(s)\|_{\hat{H}^{\sigma}(\Omega)}ds<\infty$ with $\sigma\geq 0$. Then one has
	\begin{equation*}
		\begin{aligned}
			\|\mathbf{u}-\mathbf{u}_{h,p}\|_{l^{\infty}}
			\leq Ch^{\min(\sigma+\frac{3}{2},2)}\ell_{h}^{\chi_{0}(\frac{1}{4 }-\frac{\sigma}{2})}\left (\|f(0)\|_{\hat{H}^{\sigma}(\Omega)}+\int_{0}^{t}\|f_{t}(s)\|_{\hat{H}^{\sigma}(\Omega)}ds\right )
		\end{aligned}
	\end{equation*}
	and
	\begin{equation*}
		\|\mathbf{u}-\mathbf{u}_{h,p}\|_{l^{2}}\leq Ch^{2}\left (\|f(0)\|_{\hat{H}^{\sigma}(\Omega)}+\int_{0}^{t}\|f_{t}(s)\|_{\hat{H}^{\sigma}(\Omega)}ds\right ).
	\end{equation*}
\end{theorem}

\begin{remark}
	Similarly, if $u_{0}=0$ and $\int_{0}^{t}(t-s)^{-1+\epsilon\alpha}\|f(s)\|_{\hat{H}^{\sigma}(\Omega)}ds<\infty$ with $\sigma\in[0,\frac{1}{2}]$, one has
	\begin{equation*}
		\begin{aligned}
			\|\mathbf{u}-\mathbf{u}_{h,p}\|_{l^{\infty}}
			\leq Ch^{\sigma+\frac{3}{2}-2\epsilon}\int_{0}^{t}(t-s)^{-1+\epsilon\alpha}\|f(s)\|_{\hat{H}^{\sigma}(\Omega)}ds
		\end{aligned}
	\end{equation*}
	and
	\begin{equation*}
		\begin{aligned}
			\|\mathbf{u}-\mathbf{u}_{h,p}\|_{l^{2}}
			\leq Ch^{2-2\epsilon}\int_{0}^{t}(t-s)^{-1+\epsilon\alpha}\|f(s)\|_{\hat{H}^{\sigma}(\Omega)}ds.
		\end{aligned}
	\end{equation*}
	While if $u_{0}=0$ and $\int_{0}^{t}(t-s)^{-1+\gamma\alpha}\|f(s)\|_{\hat{H}^{\sigma}(\Omega)}ds<\infty$ with $\sigma>\frac{1}{2}$ and $\gamma=\min(\frac{2\sigma-1}{8},\frac{1}{8})>0$, the spatial error estimates become
	\begin{equation*}
		\begin{aligned}
			\|\mathbf{u}-\mathbf{u}_{h,p}\|_{l^{\infty}}
			\leq Ch^{2}\int_{0}^{t}(t-s)^{-1+\gamma\alpha}\|f(s)\|_{\hat{H}^{\sigma}(\Omega)}ds
		\end{aligned}
	\end{equation*}
	and
	\begin{equation*}
		\begin{aligned}
			\|\mathbf{u}-\mathbf{u}_{h,p}\|_{l^{2}}
			\leq Ch^{2}\int_{0}^{t}(t-s)^{-1+\gamma\alpha}\|f(s)\|_{\hat{H}^{\sigma}(\Omega)}ds.
		\end{aligned}
	\end{equation*}
\end{remark}
\begin{remark}\label{Rehighdimen}
	For d-dimensional cube domain ($d\geq 1$), similar to the above discussions, we have the following error estimates when using the modified scheme to solve Eq. \eqref{eqretosol}:
	\begin{enumerate}
		\item if $u_{0}\in \hat{H}^{\sigma}(\Omega)$ and $f=0$ with $\sigma\geq 0$, then
		\begin{equation*}
			\|\mathbf{u}-\mathbf{u}_{h,p}\|_{l^{2}}\leq Ch^{2}t^{-\alpha}\|u_{0}\|_{\hat{H}^{\sigma}(\Omega)},
		\end{equation*}
		where $\mathbf{u}$ and $\mathbf{u}_{h,p}$ are the exact and numerical solutions;
		\item if $u_{0}\in \hat{H}^{\sigma}(\Omega)$ and $f=0$ with $\sigma\geq\max(0,\frac{d}{2}-2)$, then
		\begin{equation*}
			\|\mathbf{u}-\mathbf{u}_{h,p}\|_{l^{\infty}}\leq Ch^{\min(\sigma+2-\frac{d}{2},2)}\ell_{h}^{\chi_{0}(\frac{d}{4 }-\frac{\sigma}{2})}t^{-\alpha}\|u_{0}\|_{\hat{H}^{\sigma}(\Omega)},
		\end{equation*}
		where $\mathbf{u}$ and $\mathbf{u}_{h,p}$ are the exact and numerical solutions;
		\item if $u_{0}=0$, $f(0)\in \hat{H}^{\sigma}(\Omega)$, and $\int_{0}^{t}\|f_{t}(s)\|_{\hat{H}^{\sigma}(\Omega)}ds<\infty$ with $\sigma\geq 0$, then
		\begin{equation*}
			\|\mathbf{u}-\mathbf{u}_{h,p}\|_{l^{2}}\leq Ch^{2}\left (\|f(0)\|_{\hat{H}^{\sigma}(\Omega)}+\int_{0}^{t}\|f_{t}(s)\|_{\hat{H}^{\sigma}(\Omega)}ds\right ),
		\end{equation*}
		where $\mathbf{u}$ and $\mathbf{u}_{h,p}$ are the exact and numerical solutions;
		\item if $u_{0}=0$, $f(0)\in \hat{H}^{\sigma}(\Omega)$, and $\int_{0}^{t}\|f_{t}(s)\|_{\hat{H}^{\sigma}(\Omega)}ds<\infty$ with $\sigma\geq\max(0,\frac{d}{2}-2)$, then
		\begin{equation*}
			\|\mathbf{u}-\mathbf{u}_{h,p}\|_{l^{\infty}}\leq Ch^{\min(\sigma+2-\frac{d}{2},2)}\ell_{h}^{\chi_{0}(\frac{d}{4 }-\frac{\sigma}{2})}\left (\|f(0)\|_{\hat{H}^{\sigma}(\Omega)}+\int_{0}^{t}\|f_{t}(s)\|_{\hat{H}^{\sigma}(\Omega)}ds\right ),
		\end{equation*}
		where $\mathbf{u}$ and $\mathbf{u}_{h,p}$ are the exact and numerical solutions.
	\end{enumerate}
\end{remark}
\section{Fully discrete scheme and error analysis}\label{sec3}
In this section, following the idea of the averaged scheme built in \cite{Ji.2020ASCTSfTMBEGM,Zheng.2022AaLcdmftmdewwss,Zhou.2022OCRiTDT$L$LaAS}, we provide two kinds of the averaged schemes for Eq. \eqref{eqsemischl1}, i.e., averaged $L1$ ($\overline{L1}$) scheme and averaged second order backward difference ($\overline{SBD}$) scheme, and then the corresponding error analyses are also proposed. Compared with the existing discussions, the analyses presented in this paper are applicable to all $\alpha\in(0,1)$.

Let the time step $\tau=T/L$ with $L\in \mathbb{N}^{*}$ and $t_{i}=i\tau$, $i=0,1,2,\ldots,L$. Integrating Eq. \eqref{eqsemischl1M} from $t_{n-1}$ to $t_{n}$, one can get
\begin{equation*}
	\begin{aligned}
		&\frac{1}{\Gamma(1-\alpha)\tau}\int_{t_{n-1}}^{t_{n}}\left(\int_{t_{n-1}}^{t}(t-s)^{-\alpha}\partial_{s} (\mathbf{u}_{h}-\mathbf{u}^{0})ds+\sum_{k=1}^{n-1}\int_{t_{k-1}}^{t_{k}}(t-s)^{-\alpha}\partial_{s}(\mathbf{u}_{h}-\mathbf{u}^{0})ds\right)dt\\
		&\qquad\qquad\qquad\qquad\qquad\qquad\qquad+\frac{1}{\tau}\int_{t_{n-1}}^{t_{n}}\mathbf{A}\mathbf{u}_{h}(s)ds=\frac{1}{\tau}\int_{t_{n-1}}^{t_{n}}\mathbf{f}(s)ds.
	\end{aligned}
\end{equation*}
Exchanging the order of integration leads to
\begin{equation*}
	\begin{aligned}
		&\frac{1}{\Gamma(1-\alpha)\tau}\left(\int_{t_{n-1}}^{t_{n}}\int_{s}^{t_{n}}(t-s)^{-\alpha}dt\partial_{s} (\mathbf{u}_{h}-\mathbf{u}^{0})ds+\sum_{k=1}^{n-1}\int_{t_{k-1}}^{t_{k}}\int_{t_{n-1}}^{t_{n}}(t-s)^{-\alpha}dt\partial_{s}(\mathbf{u}_{h}-\mathbf{u}^{0})ds\right)\\
		&\qquad\qquad\qquad\qquad\qquad\qquad+\frac{1}{\tau}\int_{t_{n-1}}^{t_{n}}\mathbf{A}\mathbf{u}_{h}(s)ds=\frac{1}{\tau}\int_{t_{n-1}}^{t_{n}}\mathbf{f}(s)ds.
	\end{aligned}
\end{equation*}
Simple calculations result in
\begin{equation}\label{eqeqsemisch}
	\begin{aligned}
		&\frac{1}{\Gamma(2-\alpha)\tau}\left(\int_{0}^{t_{n}}(t_{n}-s)^{1-\alpha}\partial_{s} (\mathbf{u}_{h}-\mathbf{u}^{0})ds-\int_{0}^{t_{n-1}}(t_{n-1}-s)^{1-\alpha}\partial_{s}(\mathbf{u}_{h}-\mathbf{u}^{0})ds\right)\\
		&\qquad\qquad\qquad\qquad\qquad\qquad+\frac{1}{\tau}\int_{t_{n-1}}^{t_{n}}\mathbf{A}\mathbf{u}_{h}(s)ds=\frac{1}{\tau}\int_{t_{n-1}}^{t_{n}}\mathbf{f}(s)ds.
	\end{aligned}
\end{equation}
Introducing $\mathbf{u}^{-1}_{i,h}=\mathbf{u}^{0}$ ($i=1,2$), using $\frac{\tau\mathbf{A}(3\mathbf{u}_{i,h}^{n}+\mathbf{u}_{i,h}^{n-2})}{4}$ ($i=1,2$) and $\frac{\tau(\mathbf{f}^{n}+\mathbf{f}^{n-1})}{2}$ to approximate $\int_{t_{n-1}}^{t_{n}}\mathbf{A}\mathbf{u}dt$ and $\int_{t_{n-1}}^{t_{n}}\mathbf{f}(s)dt$,  and applying $L1$ and $SBD$ methods introduced in \cite{Lin.2007Fdafttde,Lubich.1996Ndeefaoaeewapmt} to discretize Eq. \eqref{eqeqsemisch}, respectively, then $\overline{L1}$ and $\overline{SBD}$ fully discrete schemes are as follows:
\begin{equation}\label{eqfullschl1}
	\frac{\sum_{k=0}^{n-1}d^{(\alpha)}_{k}(\mathbf{u}_{1,h}^{n-k}-\mathbf{u}^{0})-\sum_{k=0}^{n-2}d^{(\alpha)}_{k}(\mathbf{u}_{1,h}^{n-1-k}-\mathbf{u}^{0})}{\tau}+\frac{\mathbf{A}(3\mathbf{u}_{1,h}^{n}+\mathbf{u}_{1,h}^{n-2})}{4}=\frac{\mathbf{f}^{n-1}+\mathbf{f}^{n}}{2}
\end{equation}
and
\begin{equation}\label{eqfullschsbd}
	\frac{\sum_{k=0}^{n-1}g^{(\alpha)}_{k}(\mathbf{u}_{2,h}^{n-k}-\mathbf{u}^{0})-\sum_{k=0}^{n-2}g^{(\alpha)}_{k}(\mathbf{u}_{2,h}^{n-1-k}-\mathbf{u}^{0})}{\tau}+\frac{\mathbf{A}(3\mathbf{u}_{2,h}^{n}+\mathbf{u}_{2,h}^{n-2})}{4}=\frac{\mathbf{f}^{n-1}+\mathbf{f}^{n}}{2},
\end{equation}
where $\mathbf{u}_{1,h}^{n}$ and $\mathbf{u}_{2,h}^{n}$ are the numerical solutions at $t_{n}$ and $\mathbf{f}^{n}=\mathbf{f}(t_{n})$.
Here $d^{(\alpha)}_{k}$ and $g^{(\alpha)}_{k}$ are defined by
\begin{equation*}
	\begin{split}
		d^{(\alpha)}_{k}=&\left\{\begin{array}{ll}
			b^{(\alpha)}_{0} &  k=0,\\
			b_{k}^{(\alpha)}-b_{k-1}^{(\alpha)} & k>0,
		\end{array}\right.\\
		b^{(\alpha)}_{k}=&\tau^{1-\alpha}\frac{(k+1)^{2-\alpha}-k^{2-\alpha}}{\Gamma(3-\alpha)},
	\end{split}
\end{equation*}
and
\begin{equation*}
	\sum_{k=0}^{\infty}g^{(\alpha)}_{k}\xi^{k}=(\delta_{\tau,2}(\xi))^{\alpha-1}=\left(\frac{(1-\xi)+(1-\xi)^{2}/2}{\tau}\right)^{\alpha-1}.
\end{equation*}
Before providing the expression of the solution to Eq. \eqref{eqfullschl1}, we introduce $Li_{p}(z)$ as
\begin{equation*}
	Li_{p}(z)=\sum_{j=1}^{\infty}\frac{z^{j}}{j^{p}}.
\end{equation*}
According to \cite{Flajolet.1999SaaaoBs,Jin.2015AaotLsftsewnd}, we know that $Li_{p}(z)$ has the following properties. 
\begin{lemma}
	For $p\neq 1,2,\ldots$, the function $Li_p(e^{-z})$ satisfies the singular expansion
	\begin{equation*}
		Li_p(e^{-z})\sim \Gamma(1-p)z^{p-1}+\sum_{l=0}^{\infty}(-1)^l\varsigma(p-l)\frac{z^l}{l!} \qquad as~z\rightarrow 0,
	\end{equation*}
	where $\varsigma(z)$ denotes the Riemann zeta function.
\end{lemma}

Similar to the proofs in \cite{Jin.2015AaotLsftsewnd}, we have
\begin{lemma}\label{lemLipabconvergence}
	Let $|z|\leq \frac{\pi}{\sin(\theta)}$ with $\theta\in (\pi/2,5\pi/6)$ and $-2<p<0$. Then
	\begin{equation*}
		Li_p(e^{-z})=\Gamma(1-p)z^{p-1}+\sum_{l=0}^{\infty}(-1)^l\varsigma(p-l)\frac{z^l}{l!}
	\end{equation*}
	converges absolutely.
\end{lemma}

To obtain the representation of solution to Eq. \eqref{eqfullschl1}, we  introduce $\mathbf{v}_{1,h}^{n}=\mathbf{u}_{1,h}^{n}-\mathbf{u}^{0}$, which leads to
\begin{equation}\label{eqfullschl10}
	\frac{\sum_{k=0}^{n-1}d^{(\alpha)}_{k}\mathbf{v}_{1,h}^{n-k}-\sum_{k=0}^{n-2}d^{(\alpha)}_{k}\mathbf{v}_{1,h}^{n-1-k}}{\tau}+\frac{\mathbf{A}(3\mathbf{v}_{1,h}^{n}+\mathbf{v}_{1,h}^{n-2})}{4}=\frac{\mathbf{f}^{n-1}+\mathbf{f}^{n}}{2}-\mathbf{A}\mathbf{u}^{0}.
\end{equation}
Multiplying $\zeta^{n}$ on both sides of Eq. \eqref{eqfullschl10} and summing it from $1$ to $\infty$, one obtains
\begin{equation*}
	\begin{aligned}
		&\sum_{n=1}^{\infty}\frac{\sum_{k=0}^{n-1}d^{(\alpha)}_{k}\mathbf{v}_{1,h}^{n-k}-\sum_{k=0}^{n-2}d^{(\alpha)}_{k}\mathbf{v}_{1,h}^{n-1-k}}{\tau}\zeta^{n}+\sum_{n=1}^{\infty}\frac{\mathbf{A}(3\mathbf{v}_{1,h}^{n}+\mathbf{v}_{1,h}^{n-2})}{4}\zeta^{n}\\
		&\qquad\qquad\qquad\qquad\qquad\qquad\qquad\qquad=\sum_{n=1}^{\infty}\frac{\mathbf{f}^{n-1}+\mathbf{f}^{n}}{2}\zeta^{n}-\sum_{n=1}^{\infty}\mathbf{A}\mathbf{u}^{0}\zeta^{n}.
	\end{aligned}
\end{equation*}
Simple calculations show that
\begin{equation*}
	\frac{1-\zeta}{\tau}\left(\sum_{k=0}^{\infty}d^{(\alpha)}_{k}\zeta^{k}\right)\sum_{j=1}^{\infty}\mathbf{v}^{j}_{1,h}\zeta^{j}+\frac{3+\zeta^{2}}{4}\sum_{j=1}^{\infty}\mathbf{A}\mathbf{v}^{j}_{1,h}\zeta^{j}=\frac{1+\zeta}{2}\sum_{j=1}^{\infty}\mathbf{f}^{j}\zeta^{j}+\frac{\mathbf{f}^{0}\zeta}{2}-\sum_{j=1}^{\infty}\mathbf{A}\mathbf{u}^{0}\zeta^{j}.
\end{equation*}
According to the definitions of $d^{(\alpha)}_{k}$ and $b^{(\alpha)}_{k}$, there holds
\begin{equation*}
	\begin{aligned}
		\sum_{k=1}^{\infty}d^{(\alpha)}_{k}\zeta^{k}=&\tau^{1-\alpha}\left (\sum_{j=1}^{\infty}(b^{(\alpha)}_j-b^{(\alpha)}_{j-1})\zeta^j+b^{(\alpha)}_0\zeta^0\right )\\
		=& \tau^{1-\alpha}(1-\zeta)\sum_{j=0}^{\infty}b^{(\alpha)}_j\zeta^j=\frac{\tau^{1-\alpha}(1-\zeta)}{\Gamma(3-\alpha)} \left(\sum_{j=0}^{\infty}((j+1)^{2-\alpha}-j^{2-\alpha})\zeta^j\right)\\
		=&\frac{\tau^{1-\alpha}}{\Gamma(3-\alpha)}\frac{(1-\zeta)^{2}}{\zeta} \left(\sum_{j=0}^{\infty}j^{2-\alpha}\zeta^{j}\right)=\frac{\tau^{1-\alpha}}{\Gamma(3-\alpha)}\frac{(1-\zeta)^{2}}{\zeta} Li_{\alpha-2}(\zeta).
	\end{aligned}
\end{equation*}
Thus
\begin{equation}\label{eqfullschl11}
	\frac{(1-\zeta)^{3}\tau^{-\alpha}}{\Gamma(3-\alpha)\zeta}Li_{\alpha-2}(\zeta)\sum_{j=1}^{\infty}\mathbf{v}^{j}_{1,h}\zeta^{j}+\frac{3+\zeta^{2}}{4}\sum_{j=1}^{\infty}\mathbf{A}\mathbf{v}^{j}_{1,h}\zeta^{j}=\frac{1+\zeta}{2}\sum_{j=1}^{\infty}\mathbf{f}^{j}\zeta^{j}+\frac{\mathbf{f}^{0}\zeta}{2}-\sum_{j=1}^{\infty}\mathbf{A}\mathbf{u}^{0}\zeta^{j}.
\end{equation}
For the convenience of our analysis, introducing
\begin{equation}\label{equdefmu}
	\begin{aligned}
		\psi_{\alpha}(\zeta)=&\frac{(1-\zeta)^{2}((1-\zeta)+(1-\zeta)^{2}/2)\tau^{-\alpha}}{\Gamma(3-\alpha)\zeta}Li_{\alpha-2}(\zeta),\\
		\mu(\zeta)=&\frac{4(1-\zeta)}{(3+\zeta^{2})((1-\zeta)+(1-\zeta)^{2}/2)}=\frac{4}{(3+\zeta^{2})(3/2-\zeta/2)},\\
		\eta_{1}(\zeta)=&\frac{4\zeta}{(3+\zeta^{2})(1-\zeta)},\quad \eta_{2}(\zeta)=\frac{4(1+\zeta)}{2(3+\zeta^{2})}=\frac{2(1+\zeta)}{(3+\zeta^{2})},
	\end{aligned}
\end{equation}
we can rewrite Eq. \eqref{eqfullschl11} as
\begin{equation*}
	\mu(\zeta)\psi_{\alpha}(\zeta)\sum_{j=1}^{\infty}\mathbf{v}^{j}_{1,h}\zeta^{j}+\sum_{j=1}^{\infty}\mathbf{A}\mathbf{v}^{j}_{1,h}\zeta^{j}=\eta_{2}(\zeta)\left (\sum_{j=1}^{\infty}\mathbf{f}^{j}\zeta^{j}+\frac{\mathbf{f}^{0}\zeta}{1+\zeta}\right )-\eta_{1}(\zeta)\mathbf{A}\mathbf{u}^{0}.
\end{equation*}
Then there holds
\begin{equation*}
	\sum_{j=1}^{\infty}\mathbf{v}^{j}_{1,h}\zeta^{j}=(\mu(\zeta)\psi_{\alpha}(\zeta)+\mathbf{A})^{-1}\eta_{2}(\zeta)\left (\sum_{j=1}^{\infty}\mathbf{f}^{j}\zeta^{j}+\frac{\mathbf{f}^{0}\zeta}{1+\zeta}\right )-(\mu(\zeta)\psi_{\alpha}(\zeta)+\mathbf{A})^{-1}\eta_{1}(\zeta)\mathbf{A}\mathbf{u}^{0}.
\end{equation*}
Using Cauchy's integral theorem and doing simple calculations result in the solution of Eq. \eqref{eqfullschl1} as
\begin{equation}\label{equal1sol}
	\begin{aligned}
		\mathbf{v}^{n}_{1,h}=&\frac{\tau}{2\pi\mathbf{i}}\int_{\Gamma_{\theta}^{\tau}}e^{zt_{n}}(\mu(e^{-z\tau})\psi_{\alpha}(e^{-z\tau})+\mathbf{A})^{-1}\eta_{2}(e^{-z\tau})\left (\sum_{j=1}^{\infty}\mathbf{f}^{j}e^{-zt_{j}}+\frac{\mathbf{f}^{0}e^{-z\tau}}{1+e^{-z\tau}}\right )dz\\
		&-\frac{\tau}{2\pi\mathbf{i}}\int_{\Gamma_{\theta}^{\tau}}e^{zt_{n}}(\mu(e^{-z\tau})\psi_{\alpha}(e^{-z\tau})+\mathbf{A})^{-1}\eta_{1}(e^{-z\tau})\mathbf{A}\mathbf{u}^{0}dz,
	\end{aligned}
\end{equation}
where $\Gamma_{\theta}^{\tau}=\{z\in\Gamma_{\theta},|z|<\frac{\pi}{\tau\sin(\theta)}\}$.

Similarly, the solution of Eq. \eqref{eqfullschsbd} can be reformulated as
\begin{equation}\label{equasbdsol}
	\begin{aligned}
		\mathbf{v}^{n}_{2,h}=&\frac{\tau}{2\pi\mathbf{i}}\int_{\Gamma_{\theta}^{\tau}}e^{zt_{n}}(\mu(e^{-z\tau})(\delta_{\tau,2}(e^{-z\tau}))^{\alpha}+\mathbf{A})^{-1}\eta_{2}(e^{-z\tau})\left (\sum_{j=1}^{\infty}\mathbf{f}^{j}e^{-zt_{j}}+\frac{\mathbf{f}^{0}e^{-z\tau}}{1+e^{-z\tau}}\right )dz\\
		&-\frac{\tau}{2\pi\mathbf{i}}\int_{\Gamma_{\theta}^{\tau}}e^{zt_{n}}(\mu(e^{-z\tau})(\delta_{\tau,2}(e^{-z\tau}))^{\alpha}+\mathbf{A})^{-1}\eta_{1}(e^{-z\tau})\mathbf{A}\mathbf{u}^{0}dz,
	\end{aligned}
\end{equation}
where $\mathbf{v}_{2,h}^{n}=\mathbf{u}_{2,h}^{n}-\mathbf{u}^{0}$.

Next, to show the convergence rates of $\overline{L1}$ scheme \eqref{eqfullschl1}, the following  lemmas are needed.

\begin{lemma}\label{lempsi}
	Let $|z|\leq \frac{\pi}{\sin(\theta)\tau}$ with $\theta\in (\pi/2,5\pi/6)$. Then we have
	\begin{equation*}
		\begin{aligned}
			&|\psi_{\alpha}(e^{-z\tau})-z^{\alpha}|\leq C|z|^{\alpha+2}\tau^{2},\quad C^{-1}|z|^{\alpha}\leq|\psi_{\alpha}(e^{-z\tau})|\leq C|z|^{\alpha}.\\
		\end{aligned}
	\end{equation*}
\end{lemma}
\begin{proof}
	Simple calculations show
	\begin{equation*}
		\begin{aligned}
			&\frac{(1-e^{-z\tau})^{2}((1-e^{-z\tau})+(1-e^{-z\tau})^{2}/2)}{e^{-z\tau}}\\
			=&e^{z\tau}(1-e^{-z\tau})^{2}((1-e^{-z\tau})+(1-e^{-z\tau})^{2}/2)\\
			=&(1+z\tau+(z\tau)^{2}/2+(z\tau)^{3}/6+\mathcal{O}((z\tau)^{4})\\
			&\cdot(z\tau-(z\tau)^{2}/2+(z\tau)^{3}/6+\mathcal{O}((z\tau)^{4}))^{2}\\
			&\cdot(z\tau-(z\tau)^{3}/3+\mathcal{O}((z\tau)^{4}))\\
			=&(z\tau)^{3}+\mathcal{O}((z\tau)^{5}).
		\end{aligned}
	\end{equation*}
	Combining Lemma \ref{lemLipabconvergence}, we arrive at
	\begin{equation*}
		\begin{aligned}
			\psi_{\alpha}(e^{-z\tau})=&\tau^{-\alpha}((z\tau)^{3}+\mathcal{O}((z\tau)^{5}))\left ((z\tau)^{\alpha-3}+\sum_{l=0}^{\infty}\frac{(-1)^{l}\varsigma(\alpha-2-k)}{\Gamma(3-\alpha)}\frac{(z\tau)^{l}}{l!}\right )\\
			=&z^{\alpha}+z^{\alpha+2}\mathcal{O}(\tau^{2}),
		\end{aligned}
	\end{equation*}
	which leads to the first desired result. Similar to the discussions in \cite{Yan.2018AaotmLsftpdewnd}, we have
	\begin{equation*}
		\begin{aligned}
			&\lim_{|z\tau|\rightarrow 0}\frac{z^{\alpha}}{\psi_{\alpha}(e^{-z\tau})}
			=\lim_{|z\tau|\rightarrow 0}\frac{z^{\alpha}}{z^{\alpha}+z^{\alpha+2}\mathcal{O}(\tau^{2})}\\
			&\qquad=\lim_{|z\tau|\rightarrow 0}\frac{1}{1+\mathcal{O}(z^{2}\tau^{2})}=1.
		\end{aligned}
	\end{equation*}
	Therefore $|z|^{\alpha}\leq C|\psi_{\alpha}(e^{-z\tau})|$. At the same time, there holds
	\begin{equation*}
		\lim_{|z\tau|\rightarrow 0}\frac{\psi_{\alpha}(e^{-z\tau})}{z^{\alpha}}
		=1,
	\end{equation*}
	so $|\psi_{\alpha}(e^{-z\tau})|\leq C|z|^{\alpha}$ holds. This ends the proof.
\end{proof}

By Taylor's expansion, it's easy to check the following estimates for $\mu$, $\eta_{1}$, and $\eta_{2}$, defined in \eqref{equdefmu}.
\begin{lemma}\label{lemueta}
	For $z\in \Gamma_{\theta}^{\tau}$, one has
	\begin{equation*}
		\begin{aligned}
			&|\mu(e^{-z\tau})-1|\leq C|z|^{3}\tau^{3},\\
			&\left |\tau\eta_{1}(e^{-z\tau})-\frac{1}{z}\right |\leq C|z|\tau^{2},\\
			%&|\eta_{2}(e^{-z\tau})-1|\leq C|z|^{2}\tau^{2},\\
			&\left |\eta_{2}(e^{-z\tau})\tau\left (\frac{ e^{-z\tau}}{1-e^{-z\tau}}+\frac{e^{-z\tau}}{1+e^{-z\tau}}\right )-\frac{1}{z}\right |\leq C|z|\tau^{2},\\
			&\left |\eta_{2}(e^{-z\tau})\frac{\tau^{2} e^{-z\tau}}{(1-e^{-z\tau})^{2}}-\frac{1}{z^{2}}\right |\leq C\tau^{2}.
		\end{aligned}
	\end{equation*}
\end{lemma}

%\begin{lemma}\cite{bibid}
%	For $z\in \Gamma^{\tau}_{\theta}$ with $\theta$ close to $\frac{\pi}{2}$, we have
%	\begin{equation*}
	%		\arg\left (\frac{(1-e^{-z\tau})^{2}}{e^{-z\tau}}Li_{\alpha-2}(e^{-z\tau})\right )\in[-\frac{\pi}{2},\frac{\alpha\pi}{2}).
	%	\end{equation*}
%	and
%	\begin{equation*}
	%		(\delta_{\tau,2}(e^{-z\tau}))^{\alpha-1}\in\Sigma_{(1-\alpha)\frac{\pi}{2}+\epsilon}.
	%	\end{equation*}
%\end{lemma}
Moreover, we show $\mu(e^{-z\tau})\psi_{\alpha}(e^{-z\tau})\in\Sigma_{\theta_{0}}$ and $\mu(e^{-z\tau})(\delta_{\tau,2}(e^{-z\tau}))^{\alpha}\in\Sigma_{\theta_{0}}$ for $z\in \Gamma^{\tau}_{\theta}$, where $\theta$ is sufficiently close to $\frac{\pi}{2}$.
\begin{lemma}\label{lemtheta}
	For $z\in \Gamma^{\tau}_{\theta}$ with $\theta$ sufficiently close to $\frac{\pi}{2}$, there exists $\theta_{0}\in(\frac{\pi}{2},\pi)$ such that
	\begin{equation*}
		\mu(e^{-z\tau})\psi_{\alpha}(e^{-z\tau})\in\Sigma_{\theta_{0}}
	\end{equation*}
	and
	\begin{equation*}
		\mu(e^{-z\tau})(\delta_{\tau,2}(e^{-z\tau}))^{\alpha}\in\Sigma_{\theta_{0}}.
	\end{equation*}
\end{lemma}

\begin{proof}
	For the case $\arg(z)=\frac{\pi}{2}$, setting $\mathbf{i}\rho=z\tau$ with $\rho\in(0,\pi]$, we have
	\begin{equation*}
		\begin{aligned}
			&\frac{1-e^{-\mathbf{i}\rho}}{3+e^{-2\mathbf{i}\rho}}\\
			=&\frac{1-\cos(\rho)+\sin(\rho)\mathbf{i}}{3+\cos(2\rho)-\sin(2\rho)\mathbf{i}}\\
			=&\frac{3-3\cos(\rho)+\cos(2\rho)-\cos(2\rho)\cos(\rho)-\sin(2\rho)\sin(\rho)}{(3+\cos(2\rho))^{2}-(\sin(2\rho))^{2}}\\
			&+\frac{\sin(2\rho)-\sin(2\rho)\cos(\rho)+3\sin(\rho)+\cos(2\rho)\sin(\rho)}{(3+\cos(2\rho))^{2}-(\sin(2\rho))^{2}}\mathbf{i}\\
			=&\frac{2-4\cos(\rho)+2\cos^{2}(\rho)}{(3+\cos(2\rho))^{2}-(\sin(2\rho))^{2}}\\
			&+\frac{2\sin(\rho)\cos(\rho)+2\sin(\rho)}{(3+\cos(2\rho))^{2}-(\sin(2\rho))^{2}}\mathbf{i}.
		\end{aligned}
	\end{equation*}
	It's easy to check that for $\rho\in(0,\pi]$, one has
	\begin{equation}
		\frac{1-e^{-\mathbf{i}\rho}}{3+e^{-2\mathbf{i}\rho}}\in \left[0,\frac{\pi}{2} \right].
	\end{equation}
	From \cite{Wang.2020THTDSfSPwND}, we obtain
	\begin{equation*}
		\arg\left (\frac{(1-e^{-\mathbf{i}\rho})^{2}}{e^{-z\tau}}Li_{\alpha-2}(e^{-\mathbf{i}\rho})\right )\in \left[-\frac{\pi}{2},\frac{\alpha\pi}{2} \right),
	\end{equation*}
	which leads to 
	\begin{equation*}
		\mu(e^{-z\tau})\psi_{\alpha}(e^{-z\tau})\in \Sigma_{\theta_{0}}
	\end{equation*}
	with $\theta_{0}\in(\frac{\pi}{2},\pi)$. As for the case $\arg(z)=-\frac{\pi}{2}$, the similar arguments show that
	\begin{equation*}
		\mu(e^{-z\tau})\psi_{\alpha}(e^{-z\tau})\in \Sigma_{\theta_{0}}.
	\end{equation*}
	According to the definitions of $\mu(\zeta)$ and $\psi_{\alpha}(\zeta)$, we can get the first desired result by using the fact that $\mu(e^{-z\tau})\psi_{\alpha}(e^{-z\tau})$ depends on $z$ continuously. As for the second estimate, following the above argument and using the fact $\delta_{\tau,2}(e^{-z\tau})\in \Sigma_{\frac{\pi}{2}+\epsilon}$ for $z\in \Gamma^{\tau}_{\theta}$ with $\theta$ being sufficiently close to $\frac{\pi}{2}$ \cite{Jin.2017CohBcqffee}, one can get it.
\end{proof}
%\begin{proof}
%	Let $z\tau=z_{Re}+z_{Im}\mathbf{i}$ with $z_{Re}$ and $z_{Im}$ being the real and image parts of $z\tau$. Then there holds
%	\begin{equation*}
	%		\begin{aligned}
		%			&\frac{1-e^{-z\tau}}{3+e^{-2z\tau}}\\
		%			=&\frac{1-e^{-z_{Re}}\cos(z_{Im})+e^{-z_{Re}}\sin(z_{Im})\mathbf{i}}{3+e^{-2z_{Re}}\cos(2z_{Im})-e^{-2z_{Re}}\sin(2z_{Im})\mathbf{i}}\\
		%			=&\frac{3-3e^{-z_{Re}}\cos(z_{Im})+e^{-2z_{Re}}\cos(2z_{Im})-e^{-3z_{Re}}\cos(2z_{Im})\cos(z_{Im})-e^{-3z_{Re}}\sin(2z_{Im})\sin(z_{Im})}{(3+e^{-2z_{Re}}\cos(2z_{Im}))^{2}+(e^{-2z_{Re}}\sin(2z_{Im}))^{2}}\\
		%			&+\frac{e^{-2z_{Re}}\sin(2z_{Im})-e^{-3z_{Re}}\sin(2z_{Im})\cos(z_{Im})+3e^{-z_{Re}}\sin(z_{Im})+e^{-3z_{Re}}\cos(2z_{Im})\sin(z_{Im})}{(3+e^{-2z_{Re}}\cos(2z_{Im}))^{2}+(e^{-2z_{Re}}\sin(2z_{Im}))^{2}}\mathbf{i}\\
		%			=&\frac{3-3e^{-z_{Re}}\cos(z_{Im})+e^{-2z_{Re}}\cos(2z_{Im})-e^{-3z_{Re}}\cos(z_{Im})}{(3+e^{-2z_{Re}}\cos(2z_{Im}))^{2}+(e^{-2z_{Re}}\sin(2z_{Im}))^{2}}\\
		%			&+\frac{e^{-2z_{Re}}\sin(2z_{Im})-e^{-3z_{Re}}\sin(2z_{Im})\cos(z_{Im})+3e^{-z_{Re}}\sin(z_{Im})+e^{-3z_{Re}}\cos(2z_{Im})\sin(z_{Im})}{(3+e^{-2z_{Re}}\cos(2z_{Im}))^{2}+(e^{-2z_{Re}}\sin(2z_{Im}))^{2}}\mathbf{i}
		%		\end{aligned}
	%	\end{equation*}
%
%\end{proof}

Then we provide the error estimate when using $\overline{L1}$ scheme to solve homogeneous problem.
\begin{theorem}\label{thmfullerrorl1f0}
	Let $\mathbf{u}_{h}$ and $\mathbf{u}^{n}_{1,h}$ be the solutions of Eqs. \eqref{eqsemischl1} and \eqref{eqfullschl1}, respectively. Assume $f=0$ and $\mathbf{u}^{0}\in l^{2}$. Then it holds
	\begin{equation*}
		\|\mathbf{u}_{h}(t_{n})-\mathbf{u}^{n}_{1,h}\|_{l^{2}}\leq C\tau^{2}t_{n}^{-2}\|\mathbf{u}^{0}\|_{l^{2}}.
	\end{equation*}
\end{theorem}
\begin{proof}
	Let $\mathbf{v}_{h}(t)=\mathbf{u}_{h}(t)-\mathbf{u}^{0}$. Simple calculations show
	\begin{equation*}
		\mathbf{v}_{h}(t)=-\frac{1}{2\pi\mathbf{i}}\int_{\Gamma_{\theta}}e^{zt}(z^{\alpha}+\mathbf{A})^{-1}z^{-1}\mathbf{A}\mathbf{u}^{0}dz.
	\end{equation*}
	To give the error estimate between $\mathbf{u}_{h}(t_{n})$ and $\mathbf{u}^{n}_{1,h}$, we just need to estimate $\|\mathbf{v}_{h}(t_{n})-\mathbf{v}^{n}_{1,h}\|_{l^{2}}$, i.e.,
	\begin{equation*}
		\begin{aligned}
			&\|\mathbf{v}_{h}(t_{n})-\mathbf{v}^{n}_{1,h}\|_{l^{2}}\\
			\leq&C\Bigg\|\int_{\Gamma_{\theta}\backslash\Gamma_{\theta}^{\tau}}e^{zt_{n}}(z^{\alpha}+\mathbf{A})^{-1}z^{-1}\mathbf{A}\mathbf{u}^{0}dz\Bigg\|_{l^{2}}\\
			&+C\Bigg\|\int_{\Gamma_{\theta}^{\tau}}e^{zt_{n}}((z^{\alpha}+\mathbf{A})^{-1}z^{-1}-(\mu(e^{-z\tau})\psi_{\alpha}(e^{-z\tau})+\mathbf{A})^{-1}\mu(e^{-z\tau})\eta(e^{-z\tau}))\mathbf{A}\mathbf{u}^{0}dz\Bigg\|_{l^{2}}\\
			\leq&\uppercase\expandafter{\romannumeral1}+\uppercase\expandafter{\romannumeral2}.
		\end{aligned}
	\end{equation*}
	For $\uppercase\expandafter{\romannumeral1}$, one has
	\begin{equation*}
		\begin{aligned}
			\uppercase\expandafter{\romannumeral1}\leq&C\tau^{2}\int_{\Gamma_{\theta}\backslash\Gamma_{\theta}^{\tau}}|e^{zt_{n}}|\|\mathbf{A}(z^{\alpha}+\mathbf{A})^{-1}z\|_{l^{2}\rightarrow l^{2}}|dz|\|\mathbf{u}^{0}\|_{l^{2}}\\
			\leq& C\tau^{2}t_{n}^{-2}\|\mathbf{u}^{0}\|_{l^{2}}.
		\end{aligned}
	\end{equation*}
	As for $\uppercase\expandafter{\romannumeral2}$, we need to consider the following estimate
	\begin{equation*}
		\begin{aligned}
			&\|((z^{\alpha}+\mathbf{A})^{-1}z^{-1}-(\mu(e^{-z\tau})\psi_{\alpha}(e^{-z\tau})+\mathbf{A})^{-1}\eta_{1}(e^{-z\tau}))\mathbf{A}\|_{l^{2}\rightarrow l^{2}}\\
			\leq&\|((z^{\alpha}+\mathbf{A})^{-1}z^{-1}-(z^{\alpha}+\mathbf{A})^{-1}\eta_{1}(e^{-z\tau}))\mathbf{A}\|_{l^{2}\rightarrow l^{2}}\\
			&+\|((z^{\alpha}+\mathbf{A})^{-1}\eta_{1}(e^{-z\tau})-(\mu(e^{-z\tau})\psi_{\alpha}(e^{-z\tau})+\mathbf{A})^{-1}\eta_{1}(e^{-z\tau}))\mathbf{A}\|_{l^{2}\rightarrow l^{2}}\\
			\leq&\|\vartheta_{1}\|_{l^{2}\rightarrow l^{2}}+\|\vartheta_{2}\|_{l^{2}\rightarrow l^{2}}.
		\end{aligned}
	\end{equation*}
	According to Lemmas \ref{lemueta}, \ref{lemtheta}, and the resolvent estimate $\|(z+\mathbf{A})^{-1}\|_{l^{2}\rightarrow l^{2}}\leq C|z|^{-1}$ for $z\in\Sigma_{\theta}$ \cite{Lubich.1996Ndeefaoaeewapmt}, there exists 
	\begin{equation*}
		\|\vartheta_{1}\|_{l^{2}\rightarrow l^{2}}\leq C\tau^{2}|z|.
	\end{equation*}
	Lemmas \ref{lempsi}, \ref{lemueta}, \ref{lemtheta} and the mean value theorem lead to
	\begin{equation*}
		\|\vartheta_{2}\|_{l^{2}\rightarrow l^{2}}\leq C\tau^{2}|z|.
	\end{equation*}
	Thus
	\begin{equation*}
		\begin{aligned}
			\uppercase\expandafter{\romannumeral2}\leq&C\tau^{2}\int_{\Gamma_{\theta}\backslash\Gamma_{\theta}^{\tau}}|e^{zt_{n}}||z||dz|\|\mathbf{u}^{0}\|_{l^{2}}\\
			\leq& C\tau^{2}t_{n}^{-2}\|\mathbf{u}^{0}\|_{l^{2}}.
		\end{aligned}
	\end{equation*}
	Then the predicted result can be reached by combining the above estimates.
\end{proof}

Similar to the proofs of Theorem \ref{thmfullerrorl1f0},  the error estimate of the $\overline{SBD}$ scheme for solving homogeneous problem can be obtained.
\begin{theorem}\label{thmfullerrorSBDf0}
	Let $\mathbf{u}_{h}$ and $\mathbf{u}^{n}_{2,h}$ be the solutions of Eqs. \eqref{eqsemischl1} and \eqref{eqfullschsbd}, respectively. Assume $f=0$ and $\mathbf{u}^{0}\in l^{2}$. Then it holds
	\begin{equation*}
		\|\mathbf{u}_{h}(t_{n})-\mathbf{u}^{n}_{2,h}\|_{l^{2}}\leq C\tau^{2}t_{n}^{-2}\|\mathbf{u}^{0}\|_{l^{2}}.
	\end{equation*}
\end{theorem}
Now we propose the error estimate of the $\overline{L1}$ scheme for inhomogeneous problem.
\begin{theorem}\label{thmfullerrorl1u0}
	Let $\mathbf{u}_{h}$ and $\mathbf{u}^{n}_{1,h}$ be the solutions of Eqs. \eqref{eqsemischl1} and \eqref{eqfullschl1}, respectively. Assume  $u_{0}=0$, $\|\mathbf{f}^{0}\|_{l^{2}}<\infty$, $\|\mathbf{f}^{0}_{t}\|_{l^{2}}<\infty$, and $\int_{0}^{t_{n}}(t_{n}-s)^{\alpha-1}\|\mathbf{f}_{tt}(s)\|_{l^{2}}ds<\infty$. Then we have
	\begin{equation*}
		\|\mathbf{u}_{h}(t_{n})-\mathbf{u}^{n}_{h,1}\|_{l^{2}}\leq C\tau^{2}\left(t_{n}^{\alpha-2}\|\mathbf{f}^{0}\|_{l^{2}}+t_{n}^{\alpha-1}\|\mathbf{f}^{0}_{t}\|_{l^{2}}+\int_{0}^{t_{n}}(t_{n}-s)^{\alpha-1}\|\mathbf{f}_{tt}(s)\|_{l^{2}}ds\right).
	\end{equation*}
	Here $\mathbf{f}_{t}^{0}=\partial_{t}\mathbf{f}|_{t=0}$ and $\mathbf{f}_{tt}$ means the second derivative of $\mathbf{f}$ about $t$.
\end{theorem}

\begin{proof}
	Using $f(t)=f(0)+tf_{t}(0)+R(t)$ with $R(t)=\int_{0}^{t}(t-s)f_{tt}(s)ds$, we can rewrite \eqref{equal1sol} as
	\begin{equation*}
		\begin{aligned}		\mathbf{v}^{n}_{1,h}=&\frac{\tau}{2\pi\mathbf{i}}\int_{\Gamma_{\theta}^{\tau}}e^{zt_{n}}(\mu(e^{-z\tau})\psi_{\alpha}(e^{-z\tau})+\mathbf{A})^{-1}\eta_{2}(e^{-z\tau})\left(\sum_{j=1}^{\infty}\mathbf{f}^{0}e^{-zt_{j}}+\frac{e^{-z\tau}}{1+e^{-z\tau}}\mathbf{f}^{0}\right)dz\\
			&+\frac{\tau}{2\pi\mathbf{i}}\int_{\Gamma_{\theta}^{\tau}}e^{zt_{n}}(\mu(e^{-z\tau})\psi_{\alpha}(e^{-z\tau})+\mathbf{A})^{-1}\eta_{2}(e^{-z\tau})\sum_{j=1}^{\infty}\mathbf{f}^{0}_{t}t_{j}e^{-zt_{j}}dz\\
			&+\frac{\tau}{2\pi\mathbf{i}}\int_{\Gamma_{\theta}^{\tau}}e^{zt_{n}}(\mu(e^{-z\tau})\psi_{\alpha}(e^{-z\tau})+\mathbf{A})^{-1}\eta_{2}(e^{-z\tau})\sum_{j=1}^{\infty}R(t_{j})e^{-zt_{j}}dz.
		\end{aligned}
	\end{equation*}
	Similarly, one has
	\begin{equation*}
		\begin{aligned}
			\mathbf{v}_{h}(t)=&\frac{1}{2\pi\mathbf{i}}\int_{\Gamma_{\theta}}e^{zt}(z^{\alpha}+\mathbf{A})^{-1}z^{-1}\mathbf{f}^{0}dz\\
			&+\frac{1}{2\pi\mathbf{i}}\int_{\Gamma_{\theta}}e^{zt}(z^{\alpha}+\mathbf{A})^{-1}z^{-2}\mathbf{f}^{0}_{t}dz\\
			&+\frac{1}{2\pi\mathbf{i}}\int_{\Gamma_{\theta}}e^{zt}(z^{\alpha}+\mathbf{A})^{-1}z^{-2}\tilde{\mathbf{f}}_{tt}dz.
		\end{aligned}
	\end{equation*}
	Thus
	\begin{equation*}
		\begin{aligned}
			&\|\mathbf{v}_{h}(t_{n})-\mathbf{v}^{n}_{1,h}\|_{l^{2}}\\
			\leq& C\Bigg \|\int_{\Gamma_{\theta}}e^{zt_{n}}(z^{\alpha}+\mathbf{A})^{-1}z^{-1}\mathbf{f}^{0}dz\\
			&-\tau\int_{\Gamma_{\theta}^{\tau}}e^{zt_{n}}(\mu(e^{-z\tau})\psi_{\alpha}(e^{-z\tau})+\mathbf{A})^{-1}\eta_{2}(e^{-z\tau})\left(\sum_{j=1}^{\infty}\mathbf{f}^{0}e^{-zt_{j}}+\frac{e^{-z\tau}}{1+e^{-z\tau}}\mathbf{f}^{0}\right)dz\Bigg \|_{l^{2}}\\
			& +C\left \|\int_{\Gamma_{\theta}}e^{zt_{n}}(z^{\alpha}+\mathbf{A})^{-1}z^{-2}\mathbf{f}^{0}_{t}dz\right. \\
			& \left.
			-\tau\int_{\Gamma_{\theta}^{\tau}}e^{zt_{n}}(\mu(e^{-z\tau})\psi_{\alpha}(e^{-z\tau})+\mathbf{A})^{-1}\eta_{2}(e^{-z\tau})\sum_{j=1}^{\infty}\mathbf{f}^{0}_{t}t_{j}e^{-zt_{j}}dz\right \|_{l^{2}}\\
			& +C\left \|\frac{1}{2\pi\mathbf{i}}\int_{\Gamma_{\theta}}e^{zt_{n}}(z^{\alpha}+\mathbf{A})^{-1}z^{-2}\tilde{\mathbf{f}}_{tt}dz\right. \\ 
			& \left. -\frac{\tau}{2\pi\mathbf{i}}\int_{\Gamma_{\theta}^{\tau}}e^{zt_{n}}(\mu(e^{-z\tau})\psi_{\alpha}(e^{-z\tau})+\mathbf{A})^{-1}\eta_{2}(e^{-z\tau})\sum_{j=1}^{\infty}R(t_{j})e^{-zt_{j}}dz\right \|_{l^{2}}\\
			\leq&\uppercase\expandafter{\romannumeral1}+\uppercase\expandafter{\romannumeral2}+\uppercase\expandafter{\romannumeral3}.
		\end{aligned}
	\end{equation*}
	For the first term $\uppercase\expandafter{\romannumeral1}$, similar to the argument in Theorem \ref{thmfullerrorl1f0}, using Lemma \ref{lemueta}, one has
	\begin{equation*}
		\begin{aligned}
			\uppercase\expandafter{\romannumeral1}\leq&C\Bigg \|\int_{\Gamma_{\theta}}e^{zt_{n}}(z^{\alpha}+\mathbf{A})^{-1}z^{-1}\mathbf{f}^{0}dz\\
			&-\tau\int_{\Gamma_{\theta}^{\tau}}e^{zt_{n}}(\mu(e^{-z\tau})\psi_{\alpha}(e^{-z\tau})+\mathbf{A})^{-1}\eta_{2}(e^{-z\tau})\left(\frac{e^{-z\tau}}{1-e^{-z\tau}}+\frac{e^{-z\tau}}{1+e^{-z\tau}}\right)\mathbf{f}^{0}dz\Bigg \|_{l^{2}}\\
			\leq& C\tau^{2}t_{n}^{\alpha-2}\|\mathbf{f}^{0}\|_{l^{2}}.
		\end{aligned}
	\end{equation*}
	As for $\uppercase\expandafter{\romannumeral2}$, using the fact $\sum_{j=1}^{\infty}t_{j}\zeta^{j}=\tau\zeta(\frac{\partial}{\partial\zeta}\frac{1}{1-\zeta})=\frac{\tau\zeta}{(1-\zeta)^{2}}$ \cite{Jin.2017CohBcqffee}, we can obtain
	\begin{equation*}
		\begin{aligned}
			\uppercase\expandafter{\romannumeral2}\leq&C\left \|\int_{\Gamma_{\theta}}e^{zt_{n}}(z^{\alpha}+\mathbf{A})^{-1}z^{-2}\mathbf{f}_{t}^{0}dz \right. \\
			&
			\left. -\int_{\Gamma_{\theta}^{\tau}}e^{zt_{n}}(\mu(e^{-z\tau})\psi_{\alpha}(e^{-z\tau})+\mathbf{A})^{-1}\eta_{2}(e^{-z\tau})\left(\frac{\tau^{2}e^{-z\tau}}{1-e^{-z\tau}}\right)\mathbf{f}^{0}_{t}dz\right \|_{l^{2}}.\\
		\end{aligned}
	\end{equation*}
	Similar to the proof of Theorem \ref{thmfullerrorl1f0},	using Lemma \ref{lemueta} yields 
	\begin{equation*}
		\uppercase\expandafter{\romannumeral2}\leq C\tau^{2}t_{n}^{\alpha-1}\|\mathbf{f}^{0}_{t}\|_{l^{2}}.
	\end{equation*}
	As for the third term $	\uppercase\expandafter{\romannumeral3}$, introduce $\mathcal{E}_{\tau}=\sum_{k=0}^{\infty}\mathcal{E}^{k}_{\tau}\delta_{t_{k}}$ with $\delta_{t_{k}}$ being the Dirac-delta function at $t_{k}$ and
	\begin{equation*}
		\mathcal{E}^{n}_{\tau}=\frac{\tau}{2\pi\mathbf{i}}\int_{\Gamma_{\theta}^{\tau}}e^{zt_{n}}(\mu(e^{-z\tau})\psi_{\alpha}(e^{-z\tau})+\mathbf{A})^{-1}\eta_{2}(e^{-z\tau})dz.
	\end{equation*}
	Thus
	\begin{equation*}
		\begin{aligned}
			& \frac{\tau}{2\pi\mathbf{i}}\int_{\Gamma_{\theta}^{\tau}}e^{zt_{n}}(\mu(e^{-z\tau})\psi_{\alpha}(e^{-z\tau})+\mathbf{A})^{-1}\eta_{2}(e^{-z\tau})\sum_{j=1}^{\infty}R(t_{j})e^{-zt_{j}}dz \\
			& =(\mathcal{E}_{\tau}\ast R)(t_{n})=((\mathcal{E}_{\tau}\ast g)\ast \mathbf{f}_{tt})(t_{n}),
		\end{aligned}
	\end{equation*}
	where $g(t)=t$ and `$\ast$' means convolution. Similarly, introduce
	\begin{equation*}
		\mathcal{E}(t)=\frac{1}{2\pi\mathbf{i}}\int_{\Gamma_{\theta}}e^{zt}(z^{\alpha}+\mathbf{A})^{-1}dz.
	\end{equation*}
	By using convolution property, one  obtains
	\begin{equation*}
		\frac{1}{2\pi\mathbf{i}}\int_{\Gamma_{\theta}}e^{zt_{n}}(z^{\alpha}+\mathbf{A})^{-1}z^{-2}\tilde{\mathbf{f}}_{tt}dz=((\mathcal{E}\ast g)\ast \mathbf{f}_{tt})(t_{n}).
	\end{equation*}
	Thus one can deduce that
	\begin{equation*}
		\uppercase\expandafter{\romannumeral3}\leq C\|(\mathcal{E}\ast g)(t)-(\mathcal{E}_{\tau}\ast g)(t)\|_{l^{2}\rightarrow l^{2}}\ast\|\mathbf{f}_{tt}\|_{l^{2}}.
	\end{equation*}
	Consider $((\mathcal{E}\ast g)(t)-(\mathcal{E}_{\tau}\ast g)(t)$ with $t=t_{n}$ first. Similar to the proof of $\uppercase\expandafter{\romannumeral1}$, one has
	\begin{equation*}
		\begin{aligned}
			&\|(\mathcal{E}\ast g)(t_{n})-(\mathcal{E}_{\tau}\ast g)(t_{n})\|_{l^{2}\rightarrow l^{2}}\\
			\leq&C\left \|\frac{1}{2\pi\mathbf{i}}\int_{\Gamma_{\theta}}e^{zt_{n}}(z^{\alpha}+\mathbf{A})^{-1}z^{-2}dz \right.
			\\
			& \left.   -\int_{\Gamma_{\theta}^{\tau}}e^{zt_{n}}(\mu(e^{-z\tau})\psi_{\alpha}(e^{-z\tau})+\mathbf{A})^{-1}\eta_{2}(e^{-z\tau})\tau\sum_{k=1}^{\infty}t_{k}e^{-zt_{k}}dz\right \|_{l^{2}\rightarrow l^{2}}\\
			\leq&C\tau^{2}t_{n}^{\alpha-1}.
		\end{aligned}
	\end{equation*}
	As for $t\in (t_{n-1},t_{n})$, using Taylor's expansion, we have
	\begin{equation*}
		\begin{aligned}
			&\|(\mathcal{E}\ast g)(t)-(\mathcal{E}_{\tau}\ast g)(t)\|_{l^{2}\rightarrow l^{2}}\\
			\leq&C\|((\mathcal{E}\ast g)(t_{n})-(\mathcal{E}_{\tau}\ast g)(t_{n})\|_{l^{2}\rightarrow l^{2}}\\
			&+C\left \|(t-t_{n})((\mathcal{E}\ast 1)(t_{n})-(\mathcal{E}_{\tau}\ast 1)(t_{n}))\right \|_{l^{2}\rightarrow l^{2}}\\
			&+C\left \|\int_{t_{n}}^{t}(t-s)(\mathcal{E}(s)-\mathcal{E}_{\tau}(s))ds\right \|_{l^{2}\rightarrow l^{2}}\\
			\leq&C\tau^{2}t_{n}^{\alpha-1}+C\tau\left \|(\mathcal{E}\ast 1)(t_{n})-(\mathcal{E}_{\tau}\ast 1)(t_{n})\right \|_{l^{2}\rightarrow l^{2}}\\
			&+C\left \|\int_{t_{n}}^{t}(t-s)\mathcal{E}(s)ds\right \|_{l^{2}\rightarrow l^{2}}+\left\|\int_{t_{n}}^{t}(t-s)\mathcal{E}_{\tau}(s)ds\right \|_{l^{2}\rightarrow l^{2}}.
		\end{aligned}
	\end{equation*}
	Similar to the discussions of $\uppercase\expandafter{\romannumeral1}$, it holds
	\begin{equation*}
		\left \|(\mathcal{E}\ast 1)(t_{n})-(\mathcal{E}_{\tau}\ast 1)(t_{n})\right \|_{l^{2}\rightarrow l^{2}}\leq C\tau t^{\alpha-1}.
	\end{equation*}
	By the definitions of $\mathcal{E}(s)$ and $\mathcal{E}_{\tau}(s)$, one obtains
	\begin{equation*}
		\left \|\int_{t_{n}}^{t}(t-s)\mathcal{E}(s)ds\right \|_{l^{2}\rightarrow l^{2}}\leq C\tau^{2}t_{n}^{\alpha-1},\quad  \left\|\int_{t_{n}}^{t}(t-s)\mathcal{E}_{\tau}(s)ds\right \|_{l^{2}\rightarrow l^{2}}\leq C\tau^{2}t_{n}^{\alpha-1},
	\end{equation*}
	which leads to 
	\begin{equation*}
		\begin{aligned}
			&\|(\mathcal{E}\ast g)(t)-(\mathcal{E}_{\tau}\ast g)(t)\|_{l^{2}\rightarrow l^{2}}
			\leq C\tau^{2}t_{n}^{\alpha-1}
		\end{aligned}
	\end{equation*}
	for $t>0$. 	Thus
	\begin{equation*}
		\uppercase\expandafter{\romannumeral3}\leq C\tau^{2}\int_{0}^{t_{n}}(t_{n}-s)^{\alpha-1}\|\mathbf{f}_{tt}\|_{l^{2}}ds.
	\end{equation*}
	Collecting the above estimates,  the desired result can be reached.
\end{proof}

Similar to the arguments of Theorem \ref{thmfullerrorl1u0},  the following error estimate of the $\overline{SBD}$ scheme for inhomogeneous problem can be deduced.
\begin{theorem}\label{thmfullerrorSBDu0}
	Let $\mathbf{u}_{h}$ and $\mathbf{u}^{n}_{2,h}$ be the solutions of Eqs. \eqref{eqsemischl1} and \eqref{eqfullschsbd}, respectively. Assume  $u_{0}=0$, $\|\mathbf{f}^{0}\|_{l^{2}}<\infty$, $\|\mathbf{f}^{0}_{t}\|_{l^{2}}<\infty$, and $\int_{0}^{t_{n}}(t_{n}-s)^{\alpha-1}\|\mathbf{f}_{tt}(s)\|_{l^{2}}ds<\infty$. Then we have
	\begin{equation*}
		\|\mathbf{u}_{h}(t_{n})-\mathbf{u}^{n}_{2,h}\|_{l^{2}}\leq C\tau^{2}\left(t_{n}^{\alpha-2}\|\mathbf{f}^{0}\|_{l^{2}}+t_{n}^{\alpha-1}\|\mathbf{f}^{0}_{t}\|_{l^{2}}+\int_{0}^{t_{n}}(t_{n}-s)^{\alpha-1}\|\mathbf{f}_{tt}(s)\|_{l^{2}}ds\right).
	\end{equation*}
\end{theorem}

\begin{remark}
	In this section, we mainly build the time discretization for central finite difference scheme \eqref{eqsemischl1} and provide relative error estimates. As for the modified central difference scheme \eqref{eqsemischmod}, the similar results can be got easily.
\end{remark}

\begin{remark}
	In this section, we only measure the temporal error in $l^{2}$-norm. As for $l^{\infty}$-norm, the corresponding error estimates can be obtained similarly.
\end{remark}

\section{Numerical Experiments}\label{sec4}
In this section, some numerical examples are provided to verify the temporal and spatial convergence rates. In all numerical experiments, we take $T=1$. Caused by the unknown solution of Eq. \eqref{eqretosol}, we can measure the spatial and temporal convergence rates by
\begin{equation*}
	Rate=\frac{\ln(e_{h}/e_{h/2})}{\ln(2)},\quad Rate=\frac{\ln(e_{\tau}/e_{\tau/2})}{\ln(2)}.
\end{equation*}
Here
\begin{equation*}
	e_{h}=\|\mathbf{u}^{n}_{h}-\mathbf{u}^{n}_{h/2}\|_{\mathbb{V}},\quad e_{\tau}=\|\mathbf{u}_{\tau}-\mathbf{u}_{\tau/2}\|_{\mathbb{V}},
\end{equation*}
where $\mathbb{V}$ is $l^{2}$ or $l^{\infty}$, $\mathbf{u}^{n}_{h}$ is the solution at time $T$ with mesh size $h$, and $\mathbf{u}_{\tau}$ is the solution at time $T$ with time step size $\tau$.
\begin{example}
	In this example, we show the temporal convergence for numerically solving the homogeneous problem \eqref{eqretosol}. Here, we take
	\begin{equation*}
		u_{0}=\chi_{x>0.5}(x)
	\end{equation*}
	with $\chi_{x>0.5}$ being the characteristic function on $x>0.5$.
	To reduce the influence of spatial discretization on convergence rates, we take $h=\frac{1}{512}$. Tables \ref{tab:timeL1f0} and \ref{tab:timeSBDf0} show the errors and convergence rates of $\overline{L1}$ and $\overline{SBD}$ schemes, respectively, which agree well with the predicted results in Theorems \ref{thmfullerrorl1f0} and \ref{thmfullerrorSBDf0}.
	\begin{table}[htbp]
		\caption{The temporal errors and convergence rates of $\overline{L1}$.}
		\begin{tabular}{c|c|ccccc}
			\hline
			$\alpha$& $L$ & 32 & 64 & 128 & 256 & 512\\
			\hline
			& $l^2$ & 3.834E-06 & 9.473E-07 & 2.353E-07 & 5.862E-08 & 1.462E-08 \\
			0.2 &  & Rate & 2.0169  & 2.0095  & 2.0050  & 2.0033  \\
			& $l^{\infty}$ & 5.833E-06 & 1.441E-06 & 3.579E-07 & 8.917E-08 & 2.224E-08 \\
			&  & Rate & 2.0169  & 2.0095  & 2.0050  & 2.0034  \\
			& $l^2$ & 1.275E-05 & 3.131E-06 & 7.732E-07 & 1.918E-07 & 4.772E-08 \\
			0.8 &  & Rate & 2.0258  & 2.0176  & 2.0110  & 2.0070  \\
			& $l^{\infty}$ & 1.847E-05 & 4.537E-06 & 1.121E-06 & 2.781E-07 & 6.919E-08 \\
			&  & Rate & 2.0255  & 2.0174  & 2.0108  & 2.0069  \\
			\hline
		\end{tabular}
		\label{tab:timeL1f0}
	\end{table}
	
	\begin{table}[htbp]
		\caption{The temporal errors and convergence rates of $\overline{SBD}$.}
		\begin{tabular}{c|c|cccccc}
			\hline
			$\alpha$& $L$ & 32 & 64 & 128&256 & 512\\
			\hline
			& $l^2$ & 6.386E-06 & 1.590E-06 & 3.961E-07 & 9.885E-08 & 2.469E-08 \\
			0.4 &  & Rate & 2.0062  & 2.0046  & 2.0027  & 2.0014  \\
			& $l^{\infty}$ & 9.626E-06 & 2.396E-06 & 5.972E-07 & 1.490E-07 & 3.722E-08 \\
			&  & Rate & 2.0061  & 2.0046  & 2.0026  & 2.0014  \\
			& $l^2$ & 1.017E-05 & 2.525E-06 & 6.284E-07 & 1.567E-07 & 3.912E-08 \\
			0.6 &  & Rate & 2.0100  & 2.0067  & 2.0038  & 2.0020  \\
			& $l^{\infty}$ & 1.512E-05 & 3.753E-06 & 9.339E-07 & 2.329E-07 & 5.813E-08 \\
			&  & Rate & 2.0099  & 2.0067  & 2.0038  & 2.0020 \\
			\hline
		\end{tabular}
		\label{tab:timeSBDf0}
	\end{table}

\end{example}

\begin{example}
	In this example, we  consider temporal convergence	rates for inhomogeneous
	problem \eqref{eqretosol}. Here, the source term $f(x,t)$ is chosen as
	\begin{equation*}
		f(x,t)=(t+1)^{1.5}x.
	\end{equation*}
	We choose $h=\frac{1}{512}$ to eliminate the influence of spatial discretization on convergence rates. The errors and convergence rates of $\overline{L1}$ and $\overline{SBD}$ schemes are shown in Tables \ref{tab:timeL1u0} and \ref{tab:timeSBDu0}, respectively. It can be noted all the convergence rates are the same with the predicted ones in Theorems \ref{thmfullerrorl1u0} and \ref{thmfullerrorSBDu0}.
	\begin{table}[htbp]
		\caption{The temporal errors and convergence rates of $\overline{L1}$.}
		\begin{tabular}{c|c|ccccc}
			\hline
			$\alpha$& $L$ & 32 & 64 & 128 & 256 & 512\\
			\hline
			& $l^2$ & 3.978E-06 & 9.934E-07 & 2.482E-07 & 6.205E-08 & 1.551E-08 \\
			0.4 &  & Rate & 2.0015  & 2.0006  & 2.0003  & 2.0001  \\
			& $l^{\infty}$ & 5.542E-06 & 1.384E-06 & 3.459E-07 & 8.645E-08 & 2.161E-08 \\
			&  & Rate & 2.0015  & 2.0006  & 2.0003  & 2.0001  \\
			& $l^2$ & 3.990E-06 & 9.970E-07 & 2.492E-07 & 6.231E-08 & 1.558E-08 \\
			0.6 &  & Rate & 2.0006  & 2.0000  & 1.9999  & 2.0000  \\
			& $l^{\infty}$ & 5.560E-06 & 1.389E-06 & 3.473E-07 & 8.683E-08 & 2.171E-08 \\
			&  & Rate & 2.0006  & 2.0000  & 1.9999  & 2.0000\\
			\hline
		\end{tabular}
		\label{tab:timeL1u0}
	\end{table}
	
	\begin{table}[htbp]
		\caption{The temporal errors and convergence rates of $\overline{SBD}$.}
		\begin{tabular}{c|c|ccccc}
			\hline
			$\alpha$& $L$& 32 & 64 & 128 & 256 & 512\\
			\hline
			& $l^2$ & 4.093E-06 & 1.021E-06 & 2.549E-07 & 6.370E-08 & 1.592E-08 \\
			0.2 &  & Rate & 2.0034  & 2.0016  & 2.0008  & 2.0004  \\
			& $l^{\infty}$ & 5.703E-06 & 1.422E-06 & 3.552E-07 & 8.874E-08 & 2.218E-08 \\
			&  & Rate & 2.0034  & 2.0016  & 2.0008  & 2.0004  \\
			& $l^2$ & 3.979E-06 & 9.937E-07 & 2.484E-07 & 6.211E-08 & 1.553E-08 \\
			0.8 &  & Rate & 2.0016  & 2.0001  & 1.9999  & 2.0000  \\
			& $l^{\infty}$ & 5.548E-06 & 1.385E-06 & 3.463E-07 & 8.658E-08 & 2.165E-08 \\
			&  & Rate & 2.0016  & 2.0001  & 1.9999  & 2.0000 \\
			\hline
		\end{tabular}
		\label{tab:timeSBDu0}
	\end{table}

\end{example}

\begin{example}
	Here we take $\tau=\frac{1}{512}$. The initial value and source term are, respectively, chosen as
	\begin{equation}\label{eqexf01}
		u_{0}=(0.25^{2}-(x-0.5)^{2})^{\sigma},\quad f(x,t)=0
	\end{equation}
	and
	\begin{equation}\label{eqexf02}
		u_{0}=(x-0.5)^{-\gamma}\chi_{x>0.5},\quad f(x,t)=0
	\end{equation}
	to verify the spatial convergence rates for solving homogeneous problem \eqref{eqretosol}.
	
	Note that for the initial value \eqref{eqexf01}, 
	\begin{equation*}
		u_{0}\in\left\{
		\begin{array}{lll}
			\hat{H}^{\sigma+\frac{1}{2}-\epsilon}(\Omega), & \sigma\in (0,1);\\
			\hat{H}^{\frac{3}{2}}(\Omega), & \sigma>1,
		\end{array}\right.
	\end{equation*}
	and for the initial data \eqref{eqexf02}, $u_{0}$ belongs to $\hat{H}^{\gamma-0.5-\epsilon}(\Omega)$ with $\gamma\in(0,0.5)$.
	
	Here we first use $\overline{L1}$ method to discretize the Riemann--Liouville fractional derivative, and apply the central finite difference  and its modified scheme in spatial direction to solve the homogeneous problem \eqref{eqretosol} with data \eqref{eqexf01}. These results are presented in Tables \ref{tab:spaceL1f0} and \ref{tab:spaceL1f0mod}, which are consistent with our predicted results. Then we use the $\overline{SBD}$ method to discretize the Riemann--Liouville fractional derivative and apply the central finite difference method and its modified scheme in spatial direction to solve Eq. \eqref{eqretosol} with data \eqref{eqexf02}. All the convergence rates presented in Tables \ref{tab:spaceSBDf0} and \ref{tab:spaceSBDf0mod} are consistent with
	the theoretical predictions.
	\begin{table}[htbp]
		\caption{The spatial errors and convergence rates of $\overline{L1}$ with data \eqref{eqexf01}.}
		\begin{tabular}{c|c|ccccc}
			\hline
			($\alpha$,$\sigma$)& N & 32 & 64 & 128 & 256 & 512\\
			\hline
			& $l^2$ & 2.616E-04 & 1.048E-04 & 4.227E-05 & 1.711E-05 & 6.936E-06 \\
			(0.2,0.3) &  & Rate & 1.3197  & 1.3098  & 1.3050  & 1.3025  \\
			& $l^{\infty}$ & 3.376E-04 & 1.327E-04 	& 5.298E-05 & 2.133E-05 & 8.627E-06 \\
			&  & Rate & 1.3473  & 1.3243  & 1.3124  & 1.3062  \\
			& $l^2$ & 4.092E-05 & 1.334E-05 & 4.376E-06 & 1.439E-06 & 4.738E-07 \\
			(0.4,0.6) &  & Rate & 1.6165  & 1.6086  & 1.6046  & 1.6026  \\
			& $l^{\infty}$ & 5.696E-05 & 1.807E-05 	& 5.791E-06 & 1.869E-06 & 6.064E-07 \\
			&  & Rate & 1.6561  & 1.6420  & 1.6316  & 1.6239  \\
			& $l^2$ & 5.386E-06 & 1.430E-06 & 3.812E-07 & 1.018E-07 & 2.720E-08 \\
			(0.6,0.9) &  & Rate & 1.9129  & 1.9076  & 1.9051  & 1.9038  \\
			& $l^{\infty}$ & 8.711E-06 & 2.291E-06 & 6.032E-07 & 1.590E-07 & 4.197E-08 \\
			&  & Rate & 1.9271  & 1.9250  & 1.9234  & 1.9220  \\
			& $l^2$ & 5.09E-07 & 1.14E-07 & 2.59E-08 & 5.94E-09 & 1.38E-09 \\
			(0.8,1.2) &  & Rate & 2.1575  & 2.1405  & 2.1245  & 2.1084  \\
			& $l^{\infty}$ & 1.05E-06 & 2.43E-07 & 5.68E-08 & 1.33E-08 & 3.14E-09 \\
			&  & Rate & 2.1044  & 2.0987  & 2.0924  & 2.0858  \\
			\hline
		\end{tabular}
		\label{tab:spaceL1f0}
	\end{table}

	\begin{table}[htbp]
		\caption{The spatial errors and convergence rates of $\overline{SBD}$ with data \eqref{eqexf02}.}
		\begin{tabular}{c|c|ccccc}
			\hline
			($\alpha$,$\gamma$)& N & 32 & 64 & 128 & 256 & 512\\
			\hline
			& $l^2$ & 8.758E-03 & 6.148E-03 & 4.332E-03 & 3.057E-03 & 2.160E-03 \\
			(0.2,0.499) &  & Rate & 0.5105  & 0.5051  & 0.5027  & 0.5014  \\
			& $l^{\infty}$ & 1.540E-02 & 1.080E-02 & 7.608E-03 & 5.367E-03 & 3.791E-03 \\
			&  & Rate & 0.5112  & 0.5059  & 0.5032  & 0.5017  \\
			& $l^2$ & 5.730E-03 & 3.887E-03 & 2.647E-03 & 1.806E-03 & 1.233E-03 \\
			(0.4,0.45) &  & Rate & 0.5600  & 0.5543  & 0.5518  & 0.5505  \\
			& $l^{\infty}$ & 1.000E-02 & 6.782E-03 & 4.616E-03 & 3.148E-03 & 2.149E-03 \\
			&  & Rate & 0.5607  & 0.5551  & 0.5523  & 0.5508  \\
			& $l^2$ & 3.292E-03 & 2.156E-03 & 1.418E-03 & 9.345E-04 & 6.163E-04 \\
			(0.6,0.4) &  & Rate & 0.6103  & 0.6044  & 0.6019  & 0.6006  \\
			& $l^{\infty}$ & 5.682E-03 & 3.720E-03 & 2.445E-03 & 1.611E-03 & 1.062E-03 \\
			&  & Rate & 0.6110  & 0.6052  & 0.6024  & 0.6009  \\
			& $l^2$ & 1.12E-03 & 6.86E-04 & 4.21E-04 & 2.59E-04 & 1.59E-04 \\
			(0.8,0.3) &  & Rate & 0.7112  & 0.7049  & 0.7021  & 0.7007  \\
			& $l^{\infty}$ & 1.91E-03 & 1.16E-03 & 7.14E-04 & 4.39E-04 & 2.70E-04 \\
			&  & Rate & 0.7121  & 0.7057  & 0.7026  & 0.7010 \\
			\hline
		\end{tabular}
		\label{tab:spaceSBDf0}
	\end{table}

	\begin{table}[htbp]
		\caption{The spatial errors and convergence rates of $\overline{L1}$-modified central finite difference scheme with data \eqref{eqexf01}.}
		\begin{tabular}{c|c|ccccc}
			\hline
			($\alpha$,$\sigma$)& N & 32 & 64 & 128 & 256 & 512\\
			\hline
			& $l^2$ & 1.356E-05 & 3.340E-06 & 8.298E-07 & 2.069E-07 & 5.166E-08 \\
			(0.2,0.3) &  & Rate & 2.0217  & 2.0091  & 2.0039  & 2.0017  \\
			& $l^{\infty}$ & 2.282E-05 & 5.694E-06 & 1.422E-06 & 3.554E-07 & 8.884E-08 \\
			&  & Rate & 2.0026  & 2.0013  & 2.0006  & 2.0002  \\
			& $l^2$ & 4.305E-06 & 1.068E-06 & 2.662E-07 & 6.646E-08 & 1.661E-08 \\
			(0.4,0.6) &  & Rate & 2.0111  & 2.0044  & 2.0019  & 2.0008  \\
			& $l^{\infty}$ & 7.780E-06 & 1.945E-06 & 4.864E-07 & 1.216E-07 & 3.040E-08 \\
			&  & Rate & 1.9996  & 2.0000  & 2.0000  & 2.0000  \\
			& $l^2$ & 1.240E-06 & 3.082E-07 & 7.687E-08 & 1.920E-08 & 4.796E-09 \\
			(0.6,0.9) &  & Rate & 2.0083  & 2.0035  & 2.0016  & 2.0008  \\
			& $l^{\infty}$ & 2.371E-06 & 5.925E-07 & 1.481E-07 & 3.703E-08 & 9.257E-09 \\
			&  & Rate & 2.0006  & 2.0001  & 2.0000  & 2.0000  \\
			& $l^2$ & 2.70E-07 & 6.70E-08 & 1.67E-08 & 4.18E-09 & 1.04E-09 \\
			(0.8,1.2) &  & Rate & 2.0079  & 2.0034  & 2.0016  & 2.0007  \\
			& $l^{\infty}$ & 5.75E-07 & 1.43E-07 & 3.58E-08 & 8.95E-09 & 2.24E-09 \\
			&  & Rate & 2.0041  & 2.0010  & 2.0002  & 2.0001  \\
			\hline
		\end{tabular}
		\label{tab:spaceL1f0mod}
	\end{table}

	\begin{table}[htbp]
		\caption{The spatial errors and convergence rates of $\overline{SBD}$-modified central finite difference scheme with data \eqref{eqexf02}.}
		\begin{tabular}{c|c|ccccc}
			\hline
			($\alpha$,$\gamma$)& N & 32 & 64 & 128 & 256 & 512\\
			\hline
			& $l^2$ & 1.450E-04 & 3.781E-05 & 9.833E-06 & 2.551E-06 & 6.601E-07 \\
			(0.2,0.499) &  & Rate & 1.9390  & 1.9428  & 1.9467  & 1.9503  \\
			& $l^{\infty}$ & 6.335E-04 & 2.238E-04 & 7.907E-05 & 2.795E-05 & 9.878E-06 \\
			&  & Rate & 1.5013  & 1.5008  & 1.5006  & 1.5003  \\
			& $l^2$ & 9.628E-05 & 2.457E-05 & 6.265E-06 & 1.595E-06 & 4.053E-07 \\
			(0.4,0.45) &  & Rate & 1.9701  & 1.9717  & 1.9740  & 1.9763  \\
			& $l^{\infty}$ & 3.962E-04 & 1.352E-04 & 4.617E-05 & 1.577E-05 & 5.387E-06 \\
			&  & Rate & 1.5510  & 1.5503  & 1.5499  & 1.5495  \\
			& $l^2$ & 5.608E-05 & 1.406E-05 & 3.529E-06 & 8.858E-07 & 2.223E-07 \\
			(0.6,0.4) &  & Rate & 1.9958  & 1.9944  & 1.9943  & 1.9947  \\
			& $l^{\infty}$ & 2.119E-04 & 6.985E-05 & 2.303E-05 & 7.597E-06 & 2.507E-06 \\
			&  & Rate & 1.6014  & 1.6006  & 1.6001  & 1.5998  \\
			& $l^2$ & 2.11E-05 & 5.19E-06 & 1.28E-06 & 3.17E-07 & 7.88E-08 \\
			(0.8,0.3) &  & Rate & 2.0266  & 2.0183  & 2.0130  & 2.0095  \\
			& $l^{\infty}$ & 6.65E-05 & 2.00E-05 & 6.12E-06 & 1.89E-06 & 5.81E-07 \\
			&  & Rate & 1.7336  & 1.7073  & 1.6990  & 1.6989  \\
			\hline
		\end{tabular}
		\label{tab:spaceSBDf0mod}
	\end{table}
	
\end{example}

\begin{example}
	To  verify the spatial convergence rates for solving inhomogeneous problem \eqref{eqretosol}, we take $\tau=\frac{1}{512}$ and choose the initial value and source term as
	\begin{equation}\label{eqexu01}
		u_{0}=0,\quad f(x,t)=(t+1)^{0.5}(0.25^{2}-(x-0.5)^{2})^{\sigma}
	\end{equation}
	and
	\begin{equation}\label{eqexu02}
		u_{0}=0,\quad f(x,t)=(t+1)^{0.5}(x-0.5)^{-\gamma}\chi_{x>0.5}.
	\end{equation}
	
	Here we use $\overline{L1}$ method to discretize the Riemann--Liouville fractional derivative, and  the central finite difference  and modified central  finite difference schemes are respectively used to approximate  Laplace operator to solve the inhomogeneous problem \eqref{eqretosol} with data \eqref{eqexu02}. These results are presented in Tables \ref{tab:spaceL1u0} and \ref{tab:spaceL1u0mod}, which are consistent with our predicted results. Then we use $\overline{SBD}$ method to discretize the Riemann--Liouville fractional derivative and apply the same discretizations as before in spatial direction  to solve  \eqref{eqretosol} with data \eqref{eqexu01}. The resulting convergence rates presented in Tables \ref{tab:spaceSBDu0} and \ref{tab:spaceSBDu0mod} agree with
	the theoretical results.
	\begin{table}[htbp]
		\caption{The spatial errors and convergence rates of $\overline{L1}$ with data \eqref{eqexu02}.}
		\begin{tabular}{c|c|ccccc}
			\hline
			($\alpha$,$\gamma$)& N & 32 & 64 & 128 & 256 & 512\\
			\hline
			& $l^2$ & 1.484E-02 & 1.042E-02 & 7.340E-03 & 5.180E-03 & 3.660E-03 \\
			(0.8,0.499) &  & Rate & 0.5104  & 0.5050  & 0.5026  & 0.5014  \\
			& $l^{\infty}$ & 2.595E-02 & 1.821E-02 & 1.282E-02 & 9.048E-03 & 6.390E-03 \\
			&  & Rate & 0.5111  & 0.5058  & 0.5032  & 0.5017  \\
			& $l^2$ & 1.172E-02 & 7.948E-03 & 5.412E-03 & 3.692E-03 & 2.521E-03 \\
			(0.6,0.45) &  & Rate & 0.5602  & 0.5543  & 0.5518  & 0.5505  \\
			& $l^{\infty}$ & 2.056E-02 & 1.394E-02 & 9.485E-03 & 6.468E-03 & 4.415E-03 \\
			&  & Rate & 0.5608  & 0.5552  & 0.5523  & 0.5508  \\
			& $l^2$ & 9.242E-03 & 6.051E-03 & 3.979E-03 & 2.622E-03 & 1.729E-03 \\
			(0.4,0.4) &  & Rate & 0.6110  & 0.6047  & 0.6020  & 0.6006  \\
			& $l^{\infty}$ & 1.625E-02 & 1.064E-02 & 6.990E-03 & 4.604E-03 & 3.035E-03 \\
			&  & Rate & 0.6116  & 0.6056  & 0.6025  & 0.6009  \\
			& $l^2$ & 5.84E-03 & 3.56E-03 & 2.18E-03 & 1.34E-03 & 8.26E-04 \\
			(0.2,0.3) &  & Rate & 0.7131  & 0.7057  & 0.7025  & 0.7009  \\
			& $l^{\infty}$ & 1.03E-02 & 6.27E-03 & 3.84E-03 & 2.36E-03 & 1.45E-03 \\
			&  & Rate & 0.7135  & 0.7064  & 0.7030  & 0.7012 \\
			\hline
		\end{tabular}
		\label{tab:spaceL1u0}
	\end{table}

	\begin{table}[htbp]
		\caption{The spatial errors and convergence rates of $\overline{SBD}$ with data \eqref{eqexu01}.}
		\begin{tabular}{c|c|ccccc}
			\hline
			($\alpha$,$\sigma$)& N & 32 & 64 & 128 & 256 & 512\\
			\hline
			& $l^2$ & 4.431E-04 & 1.776E-04 & 7.164E-05 & 2.900E-05 & 1.176E-05 \\
			(0.8,0.3) &  & Rate & 1.3192  & 1.3095  & 1.3048  & 1.3024  \\
			& $l^{\infty}$ & 5.736E-04 & 2.248E-04 & 8.958E-05 & 3.603E-05 & 1.456E-05 \\
			&  & Rate & 1.3512  & 1.3275  & 1.3140  & 1.3071  \\
			& $l^2$ & 8.395E-05 & 2.736E-05 & 8.969E-06 & 2.948E-06 & 9.705E-07 \\
			(0.6,0.6) &  & Rate & 1.6174  & 1.6093  & 1.6051  & 1.6030  \\
			& $l^{\infty}$ & 1.166E-04 & 3.696E-05 & 1.183E-05 & 3.817E-06 & 1.238E-06 \\
			&  & Rate & 1.6575  & 1.6431  & 1.6325  & 1.6246  \\
			& $l^2$ & 1.545E-05 & 4.098E-06 & 1.091E-06 & 2.909E-07 & 7.766E-08 \\
			(0.4,0.9) &  & Rate & 1.9147  & 1.9094  & 1.9068  & 1.9054  \\
			& $l^{\infty}$ & 2.488E-05 & 6.532E-06 & 1.717E-06 & 4.521E-07 & 1.191E-07 \\
			&  & Rate & 1.9294  & 1.9273  & 1.9256  & 1.9241  \\
			& $l^2$ & 2.92E-06 & 6.61E-07 & 1.51E-07 & 3.48E-08 & 8.11E-09 \\
			(0.2,1.2) &  & Rate & 2.1456  & 2.1300  & 2.1161  & 2.1025  \\
			& $l^{\infty}$ & 5.98E-06 & 1.40E-06 & 3.31E-07 & 7.82E-08 & 1.86E-08 \\
			&  & Rate & 2.0911  & 2.0857  & 2.0795  & 2.0733 \\
			\hline
		\end{tabular}
		\label{tab:spaceSBDu0}
	\end{table}

	\begin{table}[htbp]
		\caption{The spatial errors and convergence rates of $\overline{L1}$-modified central finite difference scheme with data \eqref{eqexu02}.}
		\begin{tabular}{c|c|ccccc}
			\hline
			($\alpha$,$\gamma$)& N & 32 & 64 & 128 & 256 & 512\\
			\hline
			& $l^2$ & 2.406E-04 & 6.271E-05 & 1.630E-05 & 4.227E-06 & 1.093E-06 \\
			(0.8,0.499) &  & Rate & 1.9400  & 1.9436  & 1.9474  & 1.9509  \\
			& $l^{\infty}$ & 1.044E-03 & 3.685E-04 & 1.302E-04 & 4.600E-05 & 1.626E-05 \\
			&  & Rate & 1.5019  & 1.5012  & 1.5008  & 1.5005  \\
			& $l^2$ & 2.007E-04 & 5.124E-05 & 1.307E-05 & 3.327E-06 & 8.458E-07 \\
			(0.6,0.45) &  & Rate & 1.9695  & 1.9712  & 1.9736  & 1.9759  \\
			& $l^{\infty}$ & 8.331E-04 & 2.844E-04 & 9.711E-05 & 3.317E-05 & 1.133E-05 \\
			&  & Rate & 1.5507  & 1.5501  & 1.5498  & 1.5495  \\
			& $l^2$ & 1.686E-04 & 4.228E-05 & 1.062E-05 & 2.665E-06 & 6.690E-07 \\
			(0.4,0.4) &  & Rate & 1.9950  & 1.9939  & 1.9938  & 1.9943  \\
			& $l^{\infty}$ & 6.624E-04 & 2.184E-04 & 7.206E-05 & 2.378E-05 & 7.846E-06 \\
			&  & Rate & 1.6005  & 1.6000  & 1.5997  & 1.5994  \\
			& $l^2$ & 1.24E-04 & 3.03E-05 & 7.45E-06 & 1.84E-06 & 4.57E-07 \\
			(0.2,0.3) &  & Rate & 2.0307  & 2.0219  & 2.0158  & 2.0117  \\
			& $l^{\infty}$ & 4.20E-04 & 1.29E-04 & 3.98E-05 & 1.22E-05 & 3.77E-06 \\
			&  & Rate & 1.7005  & 1.7000  & 1.6997  & 1.6995  \\
			\hline
		\end{tabular}
		\label{tab:spaceL1u0mod}
	\end{table}

	\begin{table}[htbp]
		\caption{The spatial errors and convergence rates of $\overline{SBD}$-modified central finite difference scheme with data \eqref{eqexu01}.}
		\begin{tabular}{c|c|ccccc}
			\hline
			($\alpha$,$\sigma$)& N & 32 & 64 & 128 & 256 & 512\\
			\hline
			& $l^2$ & 2.282E-05 & 5.622E-06 & 1.397E-06 & 3.483E-07 & 8.698E-08 \\
			(0.8,0.3) &  & Rate & 2.0212  & 2.0089  & 2.0038  & 2.0016  \\
			& $l^{\infty}$ & 3.774E-05 & 9.418E-06 & 2.352E-06 & 5.878E-07 & 1.469E-07 \\
			&  & Rate & 2.0026  & 2.0013  & 2.0006  & 2.0002  \\
			& $l^2$ & 8.863E-06 & 2.199E-06 & 5.480E-07 & 1.368E-07 & 3.418E-08 \\
			(0.6,0.6) &  & Rate & 2.0112  & 2.0044  & 2.0019  & 2.0008  \\
			& $l^{\infty}$ & 1.633E-05 & 4.084E-06 & 1.021E-06 & 2.553E-07 & 6.382E-08 \\
			&  & Rate & 1.9996  & 2.0000  & 2.0000  & 2.0000  \\
			& $l^2$ & 3.571E-06 & 8.876E-07 & 2.214E-07 & 5.528E-08 & 1.381E-08 \\
			(0.4,0.9) &  & Rate & 2.0084  & 2.0035  & 2.0016  & 2.0008  \\
			& $l^{\infty}$ & 7.085E-06 & 1.773E-06 & 4.435E-07 & 1.109E-07 & 2.772E-08 \\
			&  & Rate & 1.9982  & 1.9995  & 1.9999  & 2.0000  \\
			& $l^2$ & 1.47E-06 & 3.65E-07 & 9.10E-08 & 2.27E-08 & 5.68E-09 \\
			(0.2,1.2) &  & Rate & 2.0081  & 2.0035  & 2.0016  & 2.0007  \\
			& $l^{\infty}$ & 3.08E-06 & 7.71E-07 & 1.93E-07 & 4.82E-08 & 1.20E-08 \\
			&  & Rate & 1.9973  & 1.9993  & 1.9998  & 2.0000 \\
			\hline
		\end{tabular}
		\label{tab:spaceSBDu0mod}
	\end{table}
\end{example}
\section{Conclusions}
This paper is mainly to develop sharp error estimates of spatial-temporal finite difference scheme for fractional sub-diffusion equation without any regularity assumption on the exact solution. We first use central finite difference to approximate the Laplace operator,  and then by modifying the initial value and source term, the spatial convergence rates are improved.  Next, we apply $\overline{L1}$ and  $\overline{SBD}$ schemes to discretize the Riemann--Liouville fractional derivative. It is worth mentioning that the spatial error analyses are independent of the regularity of the exact solution and the temporal error estimates hold for all $\alpha\in(0,1)$. Finally, the numerical experiments validate the correctness of the theoretical analyses.

%\begin{table}
%\begin{tabular}{lll}
%\hline\noalign{\smallskip}
%\end{tabular}
%\end{table}

%\begin{acknowledgements}
%\end{acknowledgements}


\bibliographystyle{spmpsci}
\bibliography{fdm}
%\begin{thebibliography}{}
%%
%% and use \bibitem to create references. Consult the Instructions
%% for authors for reference list style.
%\end{thebibliography}

\end{document}